\documentclass{amsart}
\usepackage{amsmath,amssymb,graphicx}
\usepackage[usenames]{color}
\usepackage[shortlabels]{enumitem}

\newtheorem{thm}{Theorem}
\newtheorem{lem}[thm]{Lemma}

\newtheorem{prop}[thm]{Proposition}

\theoremstyle{definition}
\newtheorem{defn}[thm]{Definition}
\newtheorem{prob}[thm]{Problem}

\newtheorem{rmk}[thm]{Remark}

\newcommand{\CPb}{\overline{\mathbb{CP}}{}^{2}}
\newcommand{\CP}{{\mathbb{CP}}{}^{2}}
\newcommand{\PP}{{\mathbb{CP}}{}^{1}}
\newcommand{\R}{\mathbb{R}}
\newcommand{\Z}{\mathbb{Z}}

\begin{document}

\title[Unique fiber sum decomposability of genus 2 Lefschetz fibrations]
{Unique fiber sum decomposability\\ of genus 2 Lefschetz fibrations\\}

\author{Jun-Yong Park}

\address{School of Mathematics, University of Minnesota, Minneapolis, MN, 55455, USA}

\email{junepark@math.umn.edu}

\date{October 7th, 2015.}



\begin{abstract}

By applying the lantern relation substitutions to the positive relation of the genus two Lefschetz fibration over $\mathbb{S}^{2}$. We show that $K3\#2\CPb$ can be rationally blown down along seven disjoint copies of the configuration $C_2$. We compute the Seiberg-Witten invariant of the resulting symplectic 4-manifolds, and show that they are symplectically minimal. We also investigate how these exotic smooth 4-manifolds constructed via lantern relation substitution method are fiber sum decomposable. Furthermore by considering all the possible decompositions for each of our decomposable exotic examples, we will find out that there is a uniquely decomposing genus 2 Lefschetz fibration which is not a self sum of the same fibration up to diffeomorphism on the indecomposable summands.

\end{abstract}

\maketitle

\section{Introduction}

\par A nice interplay between the algebra and the topology in the Lefschetz fibration of a symplectic $4$-manifold is that the topological surgery operation that generates many interesting examples of an exotic smooth 4-manifold can be performed algebraically via monodromy substitution. One of the well understood mapping class group relation in this regard is the lantern relation which corresponds to the surgical operation of rational blowdown which gives us many interesting examples of exotic smooth 4-manifolds \cite{FS1, EG}.
In Endo-Gurtas' pioneering work, after constructing an exotic smooth 4-manifold $E$ homeomorphic but not diffeomorphic to an elliptic fibration on $E(1)= \CP\#9\CPb$ in Example 5.3 \cite{EG} via the lantern relation substitutions, they pose a problem about whether the exotic smooth 4-manifold $E$ constructed via monodromy substitution is fiber sum decomposable into a nontrivial fiber sum of other Lefschetz fibrations.

\begin{prob} \cite{EG}
Does $E$ decompose into a nontrivial fiber sum of other Lefschetz fibrations? Is $E$ isomorphic to a fiber sum of two copies of Matsumoto's fibration?
\end{prob}

\par As the manifold $E$ is homeomorphic but not diffeomorphic to $E(1)$, whereas an appropriately twisted fiber sum of two copies of Matsumoto's fibration is also homeomorphic but not diffeomorphic to $E(1)$, this is an interesting problem to investigate. While we cannot answer this problem fully we will remark at the end of our article how $E$ has unique genus 2 fiber sum decomposition up to diffeomorphism on the indecomposable summands if $E$ is fiber sum decomposable. (i.e. we will rule out any other possible genus 2 fiber sum decompositions.)

\par In this article, we will improve the construction of the Akhmedov-Park's exotic smooth 4-manifolds \cite{AP} where we found six lantern relations to finding seven lantern relations and also show how some of them are fiber sum decomposable. That is we will show how simply connected, minimal symplectic $4$-manifolds $X(n)$ for $2 \leq n \leq 7$ homeomorphic but not diffeomorphic to $3\CP\# (21-n)\CPb$ for $2 \leq n \leq 7$ with $b_2^+=3$ and symplectic Kodaira dimensions $\kappa^s = 1$ for $n=2$ and $\kappa^s = 2$ for $3 \leq n \leq 7$ acquired by starting from genus 2 Lefschetz fibration on $K3\#2\CPb$ and applying a sequence of seven rational blowdowns via lantern relation substitutions are all fiber sum decomposable for $2 \leq n \leq 6$ into nontrivial fiber sum of other genus 2 Lefschetz fibrations. 

\begin{thm}[Decomposability of $X(n)$ for $2 \leq n \leq 6$]
The genus 2 Lefschetz fibrations $X(n)$ for $2 \leq n \leq 6$ are all decomposable into nontrivial fiber sum of other genus 2 Lefschetz fibrations. Namely, $X(2)$ is isomorphic to an untwisted fiber sum of Matsumoto's fibration on $\mathbb{S}^2 \times \mathbb{T}^2 \# 4 \CPb$ with Lefschetz fibration on $Z(0) = \CP \# 13\CPb$. Additionally, $X(3), X(4), X(5),X(6)$ are isomorphic to an untwisted fiber sum of Matsumoto fibration on $\mathbb{S}^2 \times \mathbb{T}^2 \# 4 \CPb$ with $Z(1), Z(2), Z(3),Z(4)$ respectively.
\end{thm}

\par Here, $Z(m)$ for $1 \leq m \leq 4$ are examples similar to Endo-Gurtas' genus 2 examples in that they are acquired by starting from genus 2 Lefschetz fibration $Z(0) = \CP \# 13\CPb$ and applying a sequence of four rational blowdowns via lantern relation substitutions.

\par After showing decomposability, we will show that the one of the decomposable example $X(2)$ which is a minimal exotic symplectic 4-manifold with the homeomorphism type of $3\CP\#19\CPb$ with $b_2^+=3$ and symplectic Kodaira dimension $\kappa^s = 1$ has the unique genus 2 fiber sum decomposition up to diffeomorphism on the indecomposable summands.

\begin{thm}[Unique decomposition of $X(2)$]
The genus 2 Lefschetz fibration $X(2)$ which has $n$ irreducible singular fibers and $s$ reducible singular fibers pair $(n,s) = (26,2)$ must decompose under the genus 2 fiber sum having the indecomposable summands of Matsumoto's fibration on $\mathbb{S}^2 \times \mathbb{T}^2 \# 4\CPb$ and the genus 2 Lefschetz fibration on $Z(0)=\CP \# 13\CPb$. Each summands are determined up to diffeomorphism.
\end{thm}

\par Accordingly, we will narrow down all the possible genus 2 fiber sum decompositions of $X(n)= Y(1) \# Y(2)$ for $3 \leq n \leq 6$ examples with $\kappa^s=2$ by the consideration on the possible $n$ irreducible singular fibers and $s$ reducible singular fibers pair $(n,s)$ for both $Y(1), Y(2)$ where both summands $Y(1), Y(2)$ are relatively minimal genus 2 Lefschetz fibrations.

\section{Preliminaries}

For the convenience of the reader we repeat the preliminary definitions and results from \cite{AP, GS} mostly without proofs, thus making our exposition self-contained.
The list of topics that need to be recalled are the mapping class groups, the Lefschetz fibrations over $\mathbb{S}^{2}$ with details on the Matsumoto's genus two fibration on $\mathbb{S}^2 \times \mathbb{T}^2 \# 4\CPb$, lantern relation substitution and its relationship with the rational blowdown operation, the symplectic Kodaira dimension and the symplectic minimality.

\subsection{Mapping Class Groups}\label{mapping}

\par Let $\Sigma_{g}$ denote a $2$-dimensional, closed, oriented, and connected Riemann surface of genus $g>0$. 

\begin{defn}
Let $Diff^{+}\left( \Sigma_{g}\right)$ denote the group of all orientation-preserving diffeomorphisms $\Sigma_{g}\rightarrow \Sigma_{g},$ and $ Diff_{0}^{+}\left(
\Sigma_{g}\right) $ be the subgroup of $Diff^{+}\left(\Sigma_{g}\right) $ consisting of all orientation-preserving diffeomorphisms $\Sigma_{g}\rightarrow \Sigma_{g}$ that are isotopic to the identity. \emph{The mapping class group} $\Gamma_g$ of $\Sigma_{g}$ is defined to be the group of isotopy classes of orientation-preserving diffeomorphisms of $\Sigma_{g}$, i.e.,
\[\Gamma_{g}=Diff^{+}\left( \Sigma_{g}\right) /Diff_{0}^{+}\left(\Sigma_{g}\right) .\]
\end{defn}

\begin{defn}
\par Let $\alpha$ be a simple closed curve on $\Sigma_{g}$. A \emph{right handed Dehn twist} $t_\alpha$ about $\alpha$ is the isotopy class of a self-diffeomophism of $\Sigma_{g}$ obtained by cutting the surface $\Sigma_{g}$ along $\alpha$ and gluing the ends back after rotating one of the ends $2\pi$ to the right. 
\end{defn}

\begin{figure}[ht]
\begin{center}
\includegraphics[scale=.73]{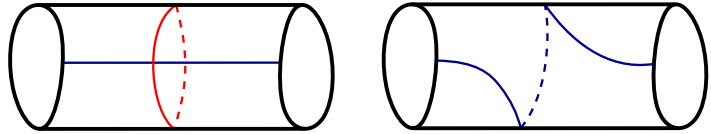}
\caption{A positive Dehn twist to a cylinder about the red curve}
\label{fig:Dehn_Twist}
\end{center}
\end{figure}

\noindent The mapping class group $\Gamma_{g}$ is finitely generated by $3g-1$ Dehn twists which was proven by the work of Dehn and Lickorish (cf. \cite{FM}). It follows that the conjugate of a Dehn twist is again a Dehn twist. That is, if $f: \Sigma_{g}\rightarrow \Sigma_{g}$ is an orientation-preserving diffeomorphism, then it is easy to check that $f \circ t_\alpha \circ f^{-1} = t_{f(\alpha)}$.

\par We will now provide a presentation for the mapping class group of the genus 2 surface $\Gamma_2$. As we will be working mostly with the genus 2 Lefscehtz fibrations restricting our attention to $\Gamma_2$ will not interfere with the construction we will illustrate.

Let $t_i$ $(i=1,\ldots, 5)$ be positive Dehn twists along the loops $c_i$ $(i=1,\ldots, 5)$ illustrated in Figure~\ref{fig:Sigma2}.  
The mapping class group $\Gamma_2$ of a genus-$2$ Riemann surface is generated by $t_1, \ldots, t_5$, and the following relations are defining relations (cf. \cite{Birman}).
{\allowdisplaybreaks 
\begin{eqnarray}
&& t_i t_j = t_j t_i     \quad \mbox{    if $|i-j|\geq 2$, } \label{eq:01} \\
&& t_i t_{i+1} t_i = t_{i+1} t_i t_{i+1}    \quad \mbox{    for $i=1, \ldots, 4$, } \label{eq:02} \\
&& \tau^2 =1    \quad \mbox{  where $\tau= t_1 t_2 t_3 t_4 t_{5}^2 t_4 t_3 t_2 t_1$,}\label{eq:03} \\
&&(t_1 t_2 t_3 t_4 t_5)^6=1, \label{eq:04} \\
&&\tau  \,  t_i = t_i \,  \tau \quad \mbox{    for $i=1, \dots, 5$.} \label{eq:05} 
\end{eqnarray}}
Let $t_\delta$ be a positive Dehn twist along the loop $\delta$ illustrated in Figure~\ref{fig:Sigma2}.  \\
Then $t_\delta= (t_1 t_2)^6$, this relation is called a chain relation.

\begin{figure}[ht]
\begin{center}
\resizebox{!}{3.5cm}{\includegraphics{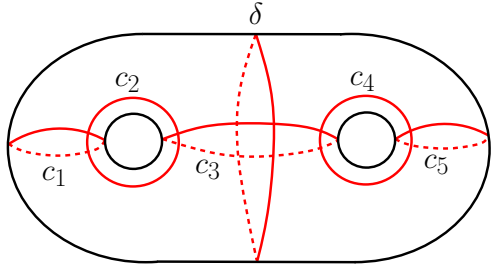}}
\caption{Curves $c_1$, $c_2$, $c_3$, $c_4$, and $c_5$}
\label{fig:Sigma2}
\end{center}
\end{figure}

\subsection{Lantern Relation}\label{lantern}

Let us recall the definition of the lantern relation which will be used extensively in our construction of exotic 4-manifolds. 

Let $\Sigma_{0,4}$ be a sphere with 4 boundary components.

\begin{lem} 
If $\delta_1, \delta_2, \delta_3, \delta_4$ are the boundary curves of $\Sigma_{0,4}$ and $\alpha$, $\beta$, $\gamma$ are 
the simple closed curves as shown in Figure~\ref{fig:lan},
then we have \[t_{\gamma }t_{\beta }t_{\alpha}=t_{\delta_{1}}t_{\delta_{2}}t_{\delta_{3}}t_{\delta_{4}},\] where $t_{\delta_{i}},$ $1\leq i\leq 4,$ denote the Dehn twists about $\delta_{i}.$
\end{lem}

For a proof see (\cite{FM, Johnson}).

\begin{figure}[ht]
\begin{center}
\includegraphics[scale=.43]{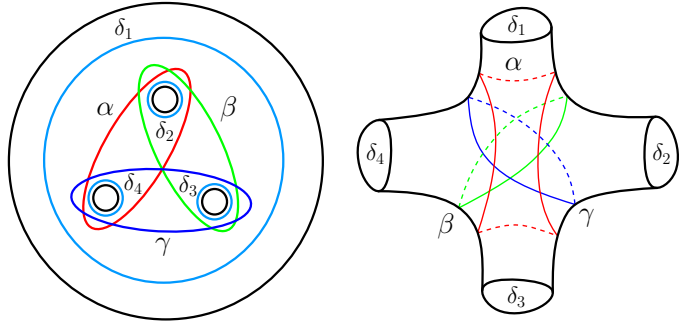}
\caption{Curves defining lantern relation drawn two ways}
\label{fig:lan}
\end{center}
\end{figure}

The lantern relation in genus 2 surface implies $t_{\gamma}t_{\beta}t_{\alpha} = t_{\beta}t_{\alpha}t_{\gamma} = t_{\alpha}t_{\gamma}t_{\beta}$. This relation follows easily from the lantern relation plus the relation that each $\delta_{i}$ for $1\leq i\leq 4$ commutes with each of $t_{\gamma}$, $t_{\beta}$, $t_{\alpha}$. Note that $t_{\gamma}t_{\beta}t_{\alpha}$ is not equal to $t_{\beta}t_{\gamma}t_{\alpha}$. We refer readers to the book of B. Farb and D. Margalit \cite{FM} for more details on mapping class group \& lantern relation.

\subsection{Lefschetz fibrations}\label{Lefschetz}

\par In this section we recall the definition of Lefschetz fibrations over $\mathbb{S}^2$ and introduce three basic examples of complex genus two fibrations with no reducible fibers. We will also introduce Matsumoto's genus two Lefschetz fibrations over $\mathbb{S}^{2}$ with $8$ singular fibers which are six irreducible fibers and two reducible fibers. They will later appear as the summands of the decomposable examples $X(n)$ for $2 \leq n \leq 6$. 

\begin{defn}\label{LF}\rm
Let $X$ be a closed, oriented smooth $4$-manifold. Lefschetz fibration of a smooth 4--manifold $X$ comprises a smooth surjective map $f : X \rightarrow \mathbb{S}^2$, which is a submersion on the complement of finitely many points $p_i$ in distinct fibers, at which there are local complex coordinates (compatible with fixed global orientations on $X$ and $\mathbb{S}^2$) with respect to which the map takes the form $(z_1, z_2) \mapsto z_1 ^2 + z_2 ^2$.  We always assume that the fibers contain no $(-1)$--spheres (``relative minimality'') so in particular the fiber genus is always strictly positive.
\end{defn}

\par By the hypotheses of good local complex models, each singular fiber of the Lefschetz fibration is a nodal curve with a unique nodal singularity, and it is obtained by shrinking a simple closed curve (the \textit{vanishing cycle}) in the regular fiber to the nodal point of the singular fiber. They fall into two classes: irreducible fibers, where we collapse a non-separating cycle in the Riemann surface, and reducible fibers, where we collapse a separating cycle which gives the one-point union of smooth Riemann surfaces of smaller genera. 

\par The existence of a Lefschetz fibration structure guarantees that $X$ is a symplectic 4--manifold with an intrinsic symplectic form which takes the shape $\omega = \tau+ N f^* \omega_{\mathbb{S}}$ where $\tau$ is a closed form which is symplectic on the smooth fibres, and $\omega_{\mathbb{S}}$ is symplectic on the base $\mathbb{S}^2 \cong \PP$. The form is symplectic for sufficiently large $N$ by the work of R. Gompf (cf. \cite{GS}). Topology of $X$ is determined by a monodromy homomorphism $\psi_X : \pi_1 (\mathbb{S}^2 \backslash \{f(p_i) \}) \ \rightarrow \ \Gamma_2$.  The map $\psi_X$ maps the generators of the fundamental group which encircle a single critical point once in an anticlockwise fashion to positive Dehn twists in the mapping class group. These Dehn twists are along the corresponding \emph{vanishing cycles}. Thus the topology of $X$ is completely encoded in an algebraic monodromy which is a word equal to the identity in the mapping class group, called a \emph{positive relation}.

\par Let $c_{1}$, $c_{2}$, $c_{3}$, $c_{4}$, and $c_{5}$ be the simple closed curves as in Figure~\ref{fig:Sigma2}. For convenience we shall denote the right handed Dehn twists $t_{c_i}$ along the curve $c_i$ by $c_{i}$. On the mapping class group $\Gamma_2$, it is well known that the following positive relations hold,

\begin{equation}
\begin{array}{l}
(c_1c_2c_3c_4{c_5}^2c_4c_3c_2c_1)^2 = 1,  \\
(c_1c_2c_3c_4c_5)^6 = 1,  \\ 
(c_1c_2c_3c_4)^{10} = 1. 
\end{array}
\end{equation}

\par For each of the positive relations above, it follows that there exists the corresponding genus 2 K\"ahler Lefschetz fibrations over $\mathbb{S}^2$ with the total spaces $\CP\#13\CPb$, $K3\#2\CPb$ and the Horikawa surface $H$ respectively. (cf. \cite{Chakiris, Smith1}).

\subsection{Matsumoto's genus two fibration}\label{m}

\par Matsumoto showed that $\mathbb{S}^2 \times \mathbb{T}^2 \# 4 \CPb$ has a genus 2 Lefschetz fibration with 6 irreducible singular fibers and 2 reducible singular fibers with a section of self-intersection -1 (cf. \cite{Matsumoto, Korkmaz}). The positive relation of the fibration is $(B_0B_1B_2\delta)^2=1$, where $B_0$, $B_1$, $B_2$, $\delta$ are the curves indicated on Figure~\ref{fig:matsumoto}.

\begin{figure}[ht]
\begin{center}
\resizebox{!}{3.5cm}{\includegraphics{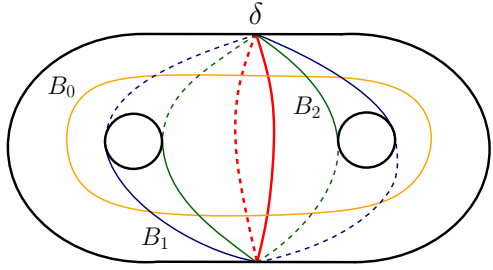}}
\caption{Curves for Matsumoto's genus 2 fibration}
\label{fig:matsumoto}
\end{center}
\end{figure}

\par By using the classfication of simple closed curves (cf. \cite{FM}), we know that there is only one nonseparating simple closed curve in surface $S$. 

\begin{prop}
If $\alpha$ and $\beta$ are any two nonseparating simple closed curves in a surface $S$, then there is a homeomorphism $\lambda : S \rightarrow S$ with $\lambda(\alpha) = \beta$.
\end{prop}

For a proof see \cite{FM}.

\par This leads to the following useful proposition proven by the work of Akhmedov and Monden \cite{AM} which is by conjugating the global monodromy for Matsumoto's genus 2 fibration by the well chosen mapping class of the $\Gamma_2$ which sends $B_0$ to $c_1$ (both are nonseparating simple closed curves) we get a positive relation that contains $(c_{1})^{2}$ which will later aid us in the construction of $X(7)$.

\begin{prop}\label{m2}
The Matsumoto's genus two Lefschetz fibration with the total space $\mathbb{S}^2 \times \mathbb{T}^2 \# 4 \CPb$ can be given by a positive relation
\begin{align}
(c_{1})^{2}(Y_{1} Y_{2} Y_{c})^{2} &= 1
\end{align}
which is acquired by conjugating global monodromy for Matsumoto's genus 2 fibration by the $\lambda = \iota \phi$ where $\phi = c_{4}^{-1} c_{3}^{-1} c_{2}^{-1} c_{1}^{-1}$ and $\iota$ is the vertical involution of the genus two surface with two fixed points. 
\end{prop}

\begin{figure}[ht]
\begin{center}
\includegraphics[scale=.48]{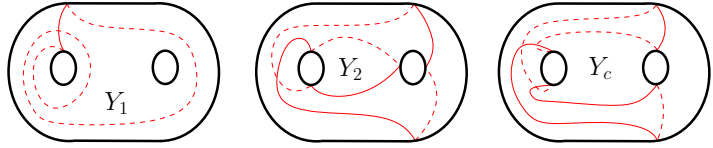}
\caption{Special curves $Y_1$, $Y_2$, $Y_c$}
\label{fig:SpecialY}
\end{center}
\end{figure}

For a detailed proof see \cite{AM}.

\subsection{Rational blowdown and Lantern relation substitution}

Surgical procedure of rational blowdown was introduced by Fintushel and Stern in 1993 \cite{FS1}  and generalized to its present form by Jongil Park in 1997 \cite{P1}
 which allowed constructions of many important examples of exotic 4-manifolds due to its explicit interplay with the Seiberg-Witten invariants. Namely, if a closed smooth $4$-manifold $X$ contains a certain configuration $C_p$ of transversally intersecting $2$-spheres whose boundary is the lens space $L(p^2, 1 - p)$ then one can construct a new smooth 4-manifold $X_p$ from $X$ by replacing the interior of $C_p$ with a rational ball $B_p$ (as $L(p^2, 1 - p)$ bounds a rational ball $B_p$ by Casson and Harer \cite{CH}) to construct a new manifold $X_p$. We say that $X_p$ is obtained by rationally blowing down $X$ along $C_p$. If one knows the Seiberg-Witten invariants of the original manifold $X$, then one can determine the Seiberg-Witten invariants of $X_p$.

\par Below we do the lightning review of the rational blowdown and refer the reader to \cite{FS1} for detailed investigation.

\par Let $p \geq  2$ and $C_p$ be the simply connected smooth $4$-manifold obtained by plumbing the $(p-1)$ disk bundles over the $2$-sphere according to the following linear diagram:

 \begin{picture}(100,60)(-90,-25)
 \put(-12,3){\makebox(200,20)[bl]{$-(p+2)$ \hspace{6pt}
                                  $-2$ \hspace{96pt} $-2$}}
 \put(4,-25){\makebox(200,20)[tl]{$u_{p-1}$ \hspace{25pt}
                                  $u_{p-2}$ \hspace{86pt} $u_{1}$}}
  \multiput(10,0)(40,0){2}{\line(1,0){40}}
  \multiput(10,0)(40,0){2}{\circle*{3}}
  \multiput(100,0)(5,0){4}{\makebox(0,0){$\cdots$}}
  \put(125,0){\line(1,0){40}}
  \put(165,0){\circle*{3}}
\end{picture}

\noindent where each node $u_{i}$ of the linear diagram represents a disk bundle over $2$-sphere with the given Euler number. 

By the work of Casson and Harer \cite{CH}, the boundary of $C_p$ is the lens space $L(p^2, 1 - p)$ which also bounds a rational ball $B_p$ with $\pi_1(B_p) = \mathbb{Z}_p$ and $\pi_1(\partial B_p) \rightarrow  \pi_1(B_p)$ surjective. If $C_p$ is embedded in a $4$-manifold $X$ then the rational blowdown manifold $X_p$ is obtained by replacing $C_p$ with $B_p$, i.e., $X_p = (X \setminus C_p) \cup B_p$. If $X$ and $X \setminus C_p$ are simply connected, then so is $X_p$. 

Note that $b_{2}^{+}(X_p) = {b_2}^{+}(X)$ so that rationally blowing down increases the signature while keeping ${b_2}^{+}$. The following is easy to check.

\begin{lem}\label{thm:rb} $b_{2}^{+}(X_p) = {b_2}^{+}(X)$, $\sigma(X_p) = \sigma(X) + (p-1)$, ${c_1}^{2}(X_p) = {c_1}^2(X) + (p-1)$, and $\chi_{h}(X_p) = \chi_{h}(X)$.
\end{lem}

\begin{proof} 

Notice that $C_{p}$ is 4-manifold with negative definite intersection form, thus we have ${b_{2}}^{+}(X_{p}) = {b_{2}}^{+}(X)$ and ${b_{2}}^{-}(X_{p}) = {b_{2}}^{-}(X) - (p-1)$. Thus,  $\sigma(X_p) = \sigma(X) + (p-1)$. Using the formulas ${c_1}^{2} = 3\sigma +2e$ and $\chi_{h} = (\sigma +e)/4$, we have ${c_{1}}^{2}(X_{p}) = 3\sigma(X_{p}) + 2e(X_{p}) = 3(\sigma(X)+(p-1)) + 2(e(X)-(p-1)) = {c_{1}}^{2}(X) + (p-1)$ and $\chi_{h}(X_p) = (\sigma(X)+(p-1) + e(X)-(p-1))/4 = \chi_{h}(X)$.  

\end{proof}

The following two theorems determines the effect of a rational blowdown on the Seiberg-Witten invariants.

\begin{thm}\label{SW1} \cite{FS1, P1}. Suppose $X$ is a smooth 4-manifold with $b_{2}^{+}(X) > 1$ which contains a configuration $C_{p}$. If $L$ is a SW basic class of $X$ satisfying $L\cdot u_{i} = 0$ for any i with $1 \leq i \leq p-2$  and $L\cdot u_{p-1} = \pm p$, then $L$ induces a SW basic class $\bar L$ of $X_{p}$ such that $SW_{X_{p}}(\bar L) = SW_{X}(L)$.  

\end{thm}

\begin{thm}\label{SW2} \cite{FS1, P1} If a simply connected smooth $4$-manifold $X$ contains a configuration $C_{p}$, then the SW-invariants of $X_{p}$ are completely determined by those of $X$. That is, for any characteristic line bundle $\bar{L}$ on $X_{p}$ with $SW_{X_{p}}(\bar{L}) \ne 0$, there exists a characteristic line bundle $L$ on $X$ such that $SW_{X}(L) = SW_{X_{p}}(\bar{L})$.

\end{thm}

In our construction we will only use the rational blowdown surgery along configuration $C_{2}$, i.e. the rational blowdowns along the $-4$ sphere. 

The following theorem of H. Endo and Y. Gurtas in 2010 \cite{EG} connects the lantern relation substitution to the rational blowdown surgical operation defined above. Namely, one can perform the topological surgery of rational blowdown via algebraic monodromy substitution in the context of Lefscehtz fibrations.

\begin{thm}\label{bd} Let $\varrho,\varrho'$ be positive relators of $\mathcal{M}_g$ 
and $M_{\varrho},M_{\varrho'}$ the corresponding Lefschetz fibrations over $S^2$, respectively.
If $\varrho'$ is obtained by applying a lantern substitution to $\varrho$, then the $4$-manifold $M_{\varrho'}$ is a rational blowdown of $M_{\varrho}$ along a configuration $C_2\subset M_{\varrho}$. 
\end{thm}

\medskip

Let us consider the following three cases of lantern substitution in $\Gamma_2$.\\

\begin{itemize}

\item Making the lantern substitution $c_{5} c_{1}^2 c_{5}$ for $\ c_{3} \delta x $.

 $ {c_{5}}c_{4}c_{3}c_{2}c_{1}c_{5}c_{4}c_{3}c_{2}{c_{1}}$ \\
$ \sim \ {}_{c_{5}}(c_{4}) \cdot c_{3}c_{2}c_{5}c_{1}{c_{5}c_{1}}c_{4}c_{3}\cdot{}_{c_{1}^{-1}}(c_{2})$ 
\\
$ \sim \ {}_{c_{5}}(c_{4}) \cdot c_{3}c_{2}\cdot {c_{5}^{2}c_{1}^{2}} \cdot c_{4}c_{3}\cdot{}_{c_{1}
^{-1}}(c_{2})$ \\
$ \overset{L} {\rightarrow} \ {}_{c_{5}}(c_{4}) \cdot c_{3}c_{2} \cdot c_{3} \delta x \cdot c_{4}c_{3}
\cdot{}_{c_{1}^{-1}}(c_{2})$ \\

\item Making the lantern substitution $c_{1} c_{3} c_{1} c_{3}$ for $ \bar{k} \bar{h} c_{5}$.

 $c_{5}c_{4}c_{3}c_{2}c_{1}{c_{5}c_{4}}c_{3}c_{2}c_{1}$ \\
$ \sim \ c_{5} c_{4} c_{5} {c_{3}} c_{2} c_{4} c_{1} c_{3} c_{2} {c_{1}}$ \\
$ \sim \ c_{5} c_{4} c_{5} \cdot {}_{c_{3}} (c_{2} c_{4}) \cdot {c_{1}^2 c_{3}^2}  \cdot {}_{c_{1}^{-1}}(c_{2}) $ \\
$ \overset{L} {\rightarrow} \ c_{5} c_{4} c_{5} \cdot {}_{c_{3}} (c_{2} c_{4}) \cdot \bar{k} \bar{h} c_{5} \cdot  {}_{c_{1}^{-1}}(c_{2}) $ \\

\item Making the lantern substitution $c_{3} c_{5}^2 c_{3}$ for $c_{1}kh$.

 ${c_{5}}c_{4}c_{3}c_{2}{c_{1}}c_{5}c_{4}{c_{3}}c_{2}c_{1}$ \\
$ \sim \ {}_{c_{5}} (c_{4}) \cdot {c_{5} c_{3}} c_{2} c_{5} c_{3} \cdot (c_{4})_{c_3} \cdot  c_{1} c_{2} c_{1}$ \\
$ \sim \ {}_{c_{5}} (c_{4}) \cdot {}_{c_{3}} (c_{2}) \cdot {c_{3}^2 c_{5}^2} \cdot {}_{c_{3}^{-1}}(c_{4}) \cdot  c_{1} c_{2} c_{1}$ \\
$ \overset{L} {\rightarrow} \ {}_{c_{5}} (c_{4}) \cdot {}_{c_{3}} (c_{2}) \cdot c_{1}  k h \cdot {}_{c_{3}^{-1}}(c_{4}) \cdot c_{1} c_{2} c_{1}$ \\
\end{itemize}

\begin{figure}[ht]
\begin{center}
\resizebox{!}{3cm}{\includegraphics{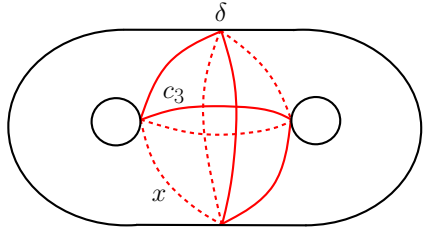}}
\caption{Special curves $x$, $\delta$}
\label{fig:Special3}
\end{center}
\end{figure}

\begin{figure}[ht]
\begin{center}
\includegraphics[scale=.41]{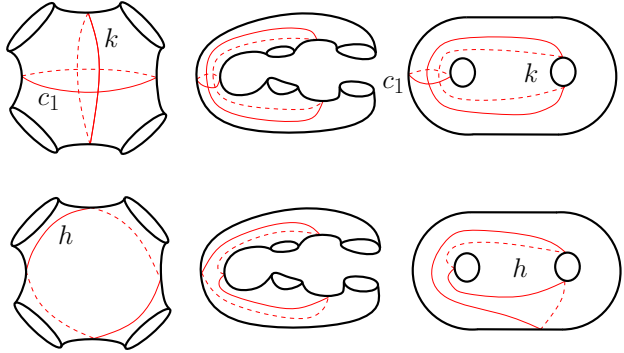}
\caption{Special curves $k$, $h$}
\label{fig:Special1}
\end{center}
\end{figure}

\begin{figure}[ht]
\begin{center}
\includegraphics[scale=.41]{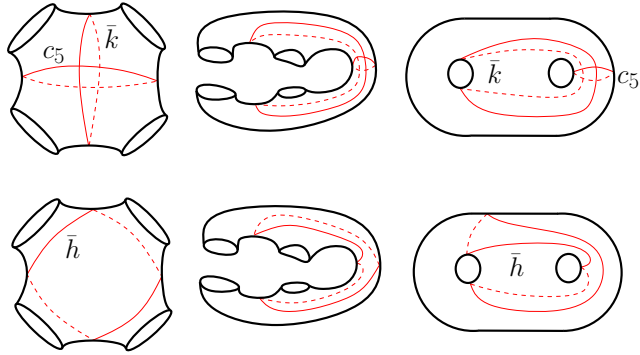}
\caption{Special curves $\bar{k}$, $\bar{h}$}
\label{fig:Special2}
\end{center}
\end{figure}

We will also need the following lemmas, which are due to R. Gompf, to analyze the symplectic $4$-manifolds constructed in Section 3. For the proof we refer the reader to \cite{Gompf, Dorf1}.

\begin{lem}\label{n=1}
Let $(X,V_X)$ be a relatively minimal smooth pair with $V_X$ an embedded $-4$ sphere.  If $X$ contains a smoothly embedded exceptional sphere transversely intersecting the hypersurface $V_X$ in a single positive point, then the manifold obtained under $-4$ blow-down of $V_X$ is diffeomorphic to the blow-down of $X$ along this sphere.
\end{lem}

\begin{lem}\label{n=2}
Let $(X,V_X)$ be a relatively minimal smooth pair with $V_X$ an embedded $-4$ sphere. If $X$ contain two disjoint smoothly embedded exceptional spheres each transversely intersecting the hypersurface $V_X$ in a single positive point, then the manifold obtained under $-4$ blow-down of $V_X$ is diffeomorphic to the blow-down of $X$ along one of these spheres.  
  
\end{lem}

\subsection{Symplectic Minimality and Symplectic Kodaira Dimension}\label{kd}

The notion of symplectic Kodaira dimension was introduced by D. McDuff and D. Salamon in 1996 \cite{MS} and discussed in detail by T.J. Li in \cite{Li}. We shall recall the definition of Kodaira dimension. In order to do so we first need to recall that a symplectic 4-manifold is called symplectically minimal if it does not contain any embedded symplectic spheres of square $-1$. For a given symplectic 4-manifold $(X,\omega)$, one can acquire its minimal model $(X',\omega')$ by blowing down a maximal disjoint collection of symplectic $(-1)$-spheres in $X$.  

\begin{defn} \label{sym Kod}
Let $(X,\omega)$ be a symplectic 4-manifold with minimal model $(X',\omega')$ and let $K_{X'} \in H^2(X';\mathbb{Z})$ denote the canonical class of $(X',\omega')$. Then the Kodaira dimension $\kappa^s(X,\omega)$ is 

$$
\kappa^s(X,\omega)=\begin{cases} \begin{array}{lll}
-\infty & \hbox{ if $K_{X'}\cdot [\omega']<0$ or} & K_{X'}\cdot K_{X'}<0,\\
0& \hbox{ if $K_{X'}\cdot [\omega']=0$ and} & K_{X'}\cdot K_{X'}=0,\\
1& \hbox{ if $K_{X'}\cdot [\omega']> 0$ and} & K_{X'}\cdot K_{X'}=0,\\
2& \hbox{ if $K_{X'}\cdot [\omega']>0$ and} & K_{X'}\cdot K_{X'}>0.\\
\end{array}
\end{cases}
$$

\end{defn}

It was shown in \cite{DZ} that the symplectic Kodaira dimension coincides with the complex Kodaira dimension when both are defined.

\section{Analysis of genus 2 Lefschetz fibrations}

We now discuss the three different ways to describe the genus 2 Lefschetz fibration on $K3\#2\CPb$ over $S^2$. The first way is to obtain the $K3\#2\CPb$ fibration as a double covering of $\mathbb{F}_1 = \CP\#\CPb$ branched along a smooth algebraic curve in the linear system $|6L|$, where $L$ is a line in $\CP$ avoiding the blown-up point. This way of thinking about the $K3\#2\CPb$ fibration is discussed in detail in Lemma 6 of Akhmedov-Park \cite{AP}. Another way is to obtain $K3\#2\CPb$ fibration by holomorphically blowing up (the ordinary blow ups) twice the genus 2 pencil where the pencil itself is in turn acquired by the identity fiber summation of two copies of elliptic fibration on $E(1)= \CP\#9\CPb$ along a regular torus fiber. Geometrically inclined readers will enjoy reading Proposition 7 of Akhmedov-Park \cite{AP} where this way of construction is given in detail. Finally, we portray here yet another way of thinking about $K3\#2\CPb$ over $\mathbb{S}^2$. This new way is to obtain $K3\#2\CPb$ fibration by rationally blowing up twice the genus 2 Lefschetz fibration where the genus 2 Lefschetz fibration itself is in turn acquired by the identity fiber summation of genus 2 Matsumoto's fibration on $\mathbb{S}^2 \times \mathbb{T}^2 \# 4 \CPb$ with genus 2 Lefschetz fibration on $\CP \# 13\CPb$.

\prop\label{twol}
The genus two Lefschetz fibration on $K3\#2\CPb$ over $\mathbb{S}^2$ can be acquired through performing two rational blowups on an untwisted fiber sum of genus two Matsumoto's fibration on $\mathbb{S}^2 \times \mathbb{T}^2 \# 4 \CPb$ with the rational genus two Lefschetz fibration on $\CP \# 13\CPb$.

\begin{proof}

Consider untwisted fiber sum (fiber sum with the identity map for the gluing diffeomorphism) of the Matsumoto's genus 2 Lefschetz fibration on $\mathbb{S}^2 \times \mathbb{T}^2 \# 4 \CPb$ given by the positive relation $(B_0B_1B_2\delta)^2 = 1$ in Section with the genus 2 rational Lefschetz fibration on $\CP \# 13\CPb$ given by the positive relation $(c_1c_2c_3c_4{c_5}^2c_4c_3c_2c_1)^2 = 1$ in Section along a generic fiber $\Sigma_2$. As the concatenation in the mapping class group corresponds to the symplectic fiber summing, the monodromy factorization of the resulting the Lefschetz fibration would be,

\begin{align}
(B_0B_1B_2\delta)^2 \cdot (c_1c_2c_3c_4{c_5}^2c_4c_3c_2c_1)^2 = 1 
\end{align}

Now, we will perform two rational blowups via lantern relation substitutions where we find $c_3 \delta x = \delta x c_3 = x c_3 \delta$ through elementary moves on the monodromy and substitute it with $c_1^2c_5^2$.\\

\par Here and subsequently, when we perform monodromy computations, we denote the lantern relation substitution by $ \overset{L} {\rightarrow} \ $, the braid relation substitution by $ \overset{B} {\rightarrow} \ $, the conjugation by $ \overset{C} {\rightarrow} \ $, and the arrangement using the commutativity by $\sim$ respectively.\\

$ (B_0 B_1 B_2 \delta)^2 \cdot (c_1c_2c_3c_4{c_5}^2c_4c_3c_2c_1)^2 = 1 $\\

$ \sim B_0 B_1 B_2 \delta B_0 B_1 B_2 \delta \cdot c_1c_2c_3c_4{c_5}^2c_4c_3c_2c_1 \cdot c_1c_2c_3c_4{c_5}^2c_4c_3c_2c_1 $\\

$ \sim B_0 B_1 \cdot {}_{c_{3}^{-1}}(x) \cdot \delta B_0 B_1 \cdot {}_{c_{3}^{-1}}(x) \cdot \delta \cdot c_1c_2c_3c_4{c_5}^2c_4c_3c_2c_1 \cdot c_1c_2c_3c_4{c_5}^2c_4c_3c_2c_1 \\ \{B_2 := {}_{c_{3}^{-1}}(x)\}$\\

$ \sim B_0 B_1 \cdot (c_1c_2c_3c_4{c_5}^2c_4c_3c_2c_1) \cdot {}_{c_{3}^{-1}}(x) \cdot \delta B_0 B_1 \cdot (c_1c_2c_3c_4{c_5}^2c_4c_3c_2c_1) \cdot {}_{c_{3}^{-1}}(x) \cdot \delta $\\ \{$(c_1c_2c_3c_4{c_5}^2c_4c_3c_2c_1)$ is central\}\\

$ \sim B_0 B_1 \cdot c_1c_2c_3c_4{c_5}^2 c_4 \cdot {}_{c_{3}}(c_2) \cdot c_1 \cdot c_3 \cdot {}_{c_{3}^{-1}}(x) \cdot \delta B_0 B_1 \cdot c_1c_2c_3c_4{c_5}^2 c_4 \cdot {}_{c_{3}}(c_2) \cdot c_1 \cdot c_3 \cdot {}_{c_{3}^{-1}}(x) \cdot \delta  $ \\

$ \sim B_0 B_1 \cdot c_1c_2c_3c_4{c_5}^2 c_4 \cdot {}_{c_{3}}(c_2) \cdot c_1 \cdot x c_3 \delta \cdot B_0 B_1 \cdot c_1c_2c_3c_4{c_5}^2 c_4 \cdot {}_{c_{3}}(c_2) \cdot c_1 \cdot x c_3 \delta $ \\

$ \overset{L} {\rightarrow} \ B_0 B_1 \cdot c_1c_2c_3c_4{c_5}^2 c_4 \cdot {}_{c_{3}}(c_2) \cdot c_1 \cdot c_1^2 c_5^2 \cdot B_0 B_1 \cdot c_1c_2c_3c_4{c_5}^2 c_4 \cdot {}_{c_{3}}(c_2) \cdot c_1 \cdot x c_3 \delta  $ \\

$ \overset{L} {\rightarrow} \ B_0 B_1 \cdot c_1c_2c_3c_4{c_5}^2 c_4 \cdot {}_{c_{3}}(c_2) \cdot c_1 \cdot c_1^2 c_5^2 \cdot B_0 B_1 \cdot c_1c_2c_3c_4{c_5}^2 c_4 \cdot {}_{c_{3}}(c_2) \cdot c_1 \cdot c_1^2 c_5^2 $ \\

$ \sim (B_0 B_1 \cdot c_1c_2c_3c_4{c_5}^2 c_4 \cdot {}_{c_{3}}(c_2) \cdot c_1 \cdot c_1^2 c_5^2)^2 = \textbf{X(0)}$ \\

Topologically, the lantern relation substitution in the direction of finding $c_3 \delta x = \delta x c_3 = x c_3 \delta$ through elementary moves on the monodromy and substituting it with $c_1^2c_5^2$ has the effect of rational blowup where one replaces the rational homology 4-ball with the tubular neighborhood of a (-4)-sphere. \cite{EG}

After two rational blowups, one arrives at the genus 2 Lefschetz fibration with the above monodromy $X(0)$ for the positive relation having 30 non-separating vanishing cycles which is transitive and has no separating vanishing cycles. (All singular fibers of genus 2 Lefschetz fibration $X(0)$ are irreducible)

By the Siebert and Tian's Theorem A in \cite{ST} on sufficient condition for holomorphicity of genus 2 Lefscetz fibrations over the $\mathbb{S}^2$. We know that this genus 2 Lefschetz fibration is isomorphic to a holomorphic genus 2 Lefschetz fibration.

As Chakiris \cite{Chakiris} assertion says every holomorphic fibrations of genus 2 without virtual reducible singular fibers is a fiber sum of three typical fibration (in our case either multiple of 20 or 30 irreducible singular fibers). This genus 2 holomorphic Lefschetz fibration with 30 irreducible singular fibers is clearly isomorphic to the fibration of $K3\#2\CPb$ with the above monodromy factorization.

\end{proof}

\rmk By using the Theorem 3.5 in Auroux's \cite{Ar1} reformulation of the holomorphicity result obtained by Siebert and Tian in terms of the mapping class group factorizations indicates $X(0)$ is Hurwitz equivalent to a factorization of the form $(c_5c_4c_3c_2c_1)^6=1$ and thus the fibration is isomorphic to the one given in Akhmedov-Park's Lemma 6 and proposition 7 \cite{AP}. \\

\par Now, we will provide propositions for the characterization of genus 2 Lefschetz fibrations with 20 irreducible singular fibers $(n,s)=(20,0)$ and 18 irreducible singular fibers and 1 reducible singular fiber $(n,s)=(18,1)$. Such characterizations of the genus 2 Lefschetz fibrations up to diffeomorphism will aid us in section 4 where we will consider all the possible decompositions of our decomposable exotic 4-manifolds examples. The proofs are adapted from Y. Sato's strategy which was effective in showing characterizations of seven and eight singular fibers genus 2 Lefschetz fibrations. (cf. \cite{Sato2})

\par Suppose that a genus 2 Lefschetz fibration $f:X \to \mathbb{S}^2$ has $n$ irreducible singular fibers and $s$ reducible singular fibers. Since the abelianization $\Gamma_2^{ab}$ of the mapping class group $\Gamma_2$ is isomorphic to $\Z/10\Z$ (cf. \cite{Matsumoto}), we have $n+2s \equiv 0 \pmod{10}$. As every singular fiber contributes 1 to the Euler characteristics $e(X)$, we have $e(X)= \# {\emph{singular fibers}} + e(\mathbb{S}^{2}) e(\Sigma_{2}) = n+s-4$. Moreover, for the signature $\sigma(X)$, we have $\sigma(X)=-3n/5-s/5$ by the Matsumoto's local signature formula \cite{Matsumoto}. 

\begin{prop}[Characterization of genus 2 Lefschetz fibration with 20 irreducible singular fibers]\label{200}
Let $f:X \to \mathbb{S}^2$ has 20 irreducible singular fibers, then X is diffeomorphic to $\CP \# 13\CPb$.
\end{prop}

\begin{proof}
Let $f:X \to \mathbb{S}^2$ be a genus 2 Lefschetz fibration with $n$ irreducible singular fibers and $s$ reducible singular fibers.
As the fibration we are interested in has $(20,0)$ for $(n,s)$ pair, its Euler characteristic and signature numbers are equal to $e(X)=16$ and $\sigma(X)=-12$ with $c_1^2(X)= 2e(X)+3\sigma(X)=-4$. Next, we will determine $(b_2^+,b_2^-,b_1)$ for $X$. Since $2-2b_1+2b_2^+=e+\sigma=4$ we obtain $b_2^+=b_1+1$. Let $H$ be the subspace of $H_1(\Sigma_2;\R)$ generated by the vanishing cycles of X. Here, $\Sigma_2$ denotes the reference fiber of genus 2. Since a Lefschetz fibration over $\mathbb{S}^2$ must have a nonseparating vanishing cycle \cite{Stipsicz1}, we have $dim \ H\geq 1$. And since $H_1(X;\R)=H_1(\Sigma_2;\R)/H$, we acquire that $b_1(X)=4-dim \ H\leq 3$. Thus, we have that $1\leq b_2^+=b_1+1\leq 4$, and therefore gives four possible triple for $(b_2^+,b_2^-,b_1)=(1,13,0),(2,14,1),(3,15,2)$ or $ (4,16,3)$. Suppose that $b_2^+> 1$. We will show this is impossible as $K_X^2=c_1^2=3\sigma+2e=-4$. Hence it follows from Theorem 0.2 in \cite{Taubes2} that $X$ is not minimal, that is, $f:X \to \mathbb{S}^2$ is a non-minimal genus 2 Lefschetz fibration with $(n,s)=(20,0)$. However, by the Table 1, of the geography of non-minimal genus 2 Lefschetz fibrations over $\mathbb{S}^2$ \cite{Sato1}, there is not any $b_2^+> 1$ non-minimal genus 2 Lefschetz fibration over $\mathbb{S}^2$ with  $(n,s)=(20,0)$. Therefore, a genus 2 Lefschetz fibration $f:X \to \mathbb{S}^2$ with  $(n,s)=(20,0)$ satisfies $(b_2^+,b_2^-,b_1)=(1,13,0)$.

Next we will show $X$ is a rational surface. Suppose that $X$ is not a rational surface. Let $\tilde{X}$ be the minimal model of $X$. Since  $b_2^+(\tilde{X})=1$ and  $b_1(\tilde{X})=0$, we have that $c_1^2(\tilde{X}) = 3\sigma(\tilde{X}) + 2e(\tilde{X}) = 5b_2^+(\tilde{X})-b_2^-(\tilde{X}) - 4b_1(\tilde{X}) + 4  = 9-b_2^-(\tilde{X})$. Moreover, since $\tilde{X}$ is a minimal symplectic 4-manifold with $b_2^+=1$ and $\tilde{X}$ is not rational nor ruled, it follows from \cite{Liu1} that $\tilde{X}$ satisfies $c_1^2(\tilde{X}) \geq 0 $. Hence, we have $b_2^-(\tilde{X}) \leq 9$. Since $X$ is not rational nor ruled and $X$ admits a genus 2 Lefchetz fibration over $\mathbb{S}^2$, it follows from Theorem 3.1 \cite{Sato1} that $X$ contains at most two 2-spheres with self-intersection number -1 essentially. Therefore, we have that $b_2^-(X) \leq 11$. This is in contradiction with $b_2^-(X)=13$. Thus, $X$ is a rational surface, and $X$ is diffeomorphic to $\CP\#13\CPb$. 
\end{proof}

\begin{prop}[Characterization of genus 2 Lefschetz fibration with 18 irreducible singular fibers and 1 reducible singular fiber]\label{181}
Let $f:X \to \mathbb{S}^2$ has 18 irreducible singular fibers and one reducible singular fiber, then X is diffeomorphic to $\CP \# 12\CPb$.
\end{prop}

\begin{proof}
Let $f:X \to \mathbb{S}^2$ be a genus 2 Lefschetz fibration with $n$ irreducible singular fibers and $s$ reducible singular fibers.
As the fibration we are interested in has $(18,1)$ for $(n,s)$ pair, its Euler characteristic and signature numbers are equal to $e(X)=15$ and $\sigma(X)=-11$ with $c_1^2(X)= 2e(X)+3\sigma(X)=-3$. We note that $X$ is non-spin as there is a reducible fiber (i.e. $s=1$) \cite{Stipsicz2}. Next, we will determine $(b_2^+,b_2^-,b_1)$ for $X$. Since $2-2b_1+2b_2^+=e+\sigma=4$ we obtain $b_2^+=b_1+1$. Let $H$ be the subspace of $H_1(\Sigma_2;\R)$ generated by the vanishing cycles of X. Here, $\Sigma_2$ denotes the reference fiber of genus 2. Since a Lefschetz fibration over $\mathbb{S}^2$ must have a nonseparating vanishing cycle \cite{Stipsicz1}, we have $dim \ H\geq 1$. And since $H_1(X;\R)=H_1(\Sigma_2;\R)/H$, we acquire that $b_1(X)=4-dim \ H\leq 3$. Thus, we have that $1\leq b_2^+=b_1+1\leq 4$, and therefore gives four possible triple for $(b_2^+,b_2^-,b_1)=(1,12,0),(2,13,1),(3,14,2)$ or $ (4,15,3)$. Suppose that $b_2^+> 1$. We will show this is impossible as $K_X^2=c_1^2=3\sigma+2e=-3$. Hence it follows from Theorem 0.2 in \cite{Taubes2} that $X$ is not minimal, that is, $f:X \to \mathbb{S}^2$ is a non-minimal genus 2 Lefschetz fibration with $(n,s)=(18,1)$. However, by the Table 1, of the geography of non-minimal genus 2 Lefschetz fibrations over $\mathbb{S}^2$ \cite{Sato1}, there is not any $b_2^+> 1$ non-minimal genus 2 Lefschetz fibration over $\mathbb{S}^2$ with  $(n,s)=(18,1)$. Therefore, a genus 2 Lefschetz fibration $f:X \to \mathbb{S}^2$ with  $(n,s)=(18,1)$ satisfies $(b_2^+,b_2^-,b_1)=(1,12,0)$.

Next we will show $X$ is a rational surface. Suppose that $X$ is not a rational surface. Let $\tilde{X}$ be the minimal model of $X$. Since  $b_2^+(\tilde{X})=1$ and  $b_1(\tilde{X})=0$, we have that $c_1^2(\tilde{X}) = 3\sigma(\tilde{X}) + 2e(\tilde{X}) = 5b_2^+(\tilde{X})-b_2^-(\tilde{X}) - 4b_1(\tilde{X}) + 4  = 9-b_2^-(\tilde{X})$. Moreover, since $\tilde{X}$ is a minimal symplectic 4-manifold with $b_2^+=1$ and $\tilde{X}$ is not rational nor ruled, it follows from \cite{Liu1} that $\tilde{X}$ satisfies $c_1^2(\tilde{X}) \geq 0 $. Hence, we have $b_2^-(\tilde{X}) \leq 9$. Since $X$ is not rational nor ruled and $X$ admits a genus 2 Lefchetz fibration over $\mathbb{S}^2$, it follows from Theorem 3.1 \cite{Sato1} that $X$ contains at most two 2-spheres with self-intersection number -1 essentially. Therefore, we have that $b_2^-(X) \leq 11$. This is in contradiction with $b_2^-(X)=12$. Thus, $X$ is a rational surface, and $X$ is diffeomorphic to $\CP\#12\CPb$. 
\end{proof}

\section{Construction of Decomposable Exotic 4-manifolds}

In this section, we construct simply-connected, minimal symplectic $4$-manifolds $X(n)$ for $2 \leq n \leq 6$ homeomorphic but not diffeomorphic to $3\CP\# (21-n)\CPb$ by starting from $K3\#2\CPb$ and applying a sequence of six rational blowdowns via lantern relation substitutions. These $X(n)$ are constructed in terms of methodology similar to Akhmedov-Park examples in \cite{AP} and thus shares geometric properties but we will use different monodromy and Hurwitz moves which will help us to show the decomposability of $X(n)$.

\begin{thm}[Construction of $X(n)$ for $2 \leq n \leq 6$]\label{theorem1} 
Let $X(0)$ denote the total space of the genus two Lefschetz fibration on $K3\#2\CPb$ over $S^{2}$ given by the positive relation $(B_0 B_1 \cdot c_1c_2c_3c_4{c_5}^2 c_4 \cdot {}_{c_{3}}(c_2) \cdot c_1 \cdot c_1^2 c_5^2)^2=1$ in $M_2$. There exist irreducible, simply-connected, symplectic $4$-manifolds $X(n)$ for $2 \leq n \leq 6$ which are homeomorphic but not diffeomorphic to $3\CP\# (21-n)\CPb$ for $2 \leq n \leq 6$ that can be obtained by applying six lantern substitutions to the global monodromy relation of $X(0)$. Moreover, $X(n)$ for $2 \leq n \leq 6$ are minimal symplectic $4$-manifolds with the symplectic Kodaira dimension $\kappa^s(X(2)) = 1$ and $\kappa^s(X(n)) = 2$ for $3 \leq n \leq 6$.

\end{thm}

To prove this theorem, we need to prove the following two lemmas first.

\begin{lem}[Monodromy of $Z(m)$ for $1 \leq m \leq 4$] \label{fourl} The global monodromy of genus 2 Lefschetz fibration on $\CP \# 13\CPb$ over $S^{2}$ given by the relation $Z(0) = (c_1c_2c_3c_4{c_5}^2c_4c_3c_2c_1)^2 = 1$ can be braid substituted to contain four lantern relations.
\end{lem}

\begin{proof}

We start with the identity word: $(c_1c_2c_3c_4{c_5}^2c_4c_3c_2c_1)^2 = 1$ \\

$Z(0) = (c_1c_2c_3c_4{c_5}^2c_4c_3c_2c_1)^2 = 1 $    \\

$\sim  c_1c_2c_3c_4c_5c_5c_4c_3c_2c_1c_1c_2c_3c_4c_5c_5c_4c_3c_2c_1 $    \\

$\sim  c_1c_2c_3c_4c_5c_5c_4 \cdot {}_{c_{3}}(c_2) \cdot c_1^2c_3^2 \cdot {}_{c_{3}^{-1}}(c_2) \cdot c_4c_5c_5c_4c_3c_2c_1 $   \\

$ \overset{L} {\rightarrow} \  c_1c_2c_3c_4c_5c_5c_4 \cdot {}_{c_{3}}(c_2) \cdot c_5 \bar{k} \bar{h} \cdot {}_{c_{3}^{-1}}(c_2) \cdot c_4c_5c_5c_4c_3c_2c_1=\textbf{Z(1)} $   \\

$\sim c_1c_2c_3c_4c_5c_5  c_4 \cdot c_1 \cdot {}_{c_{3}c_{1}^{-1}}(c_2) \cdot c_5 \cdot {}_{c_{1}^{-1}}(\bar{k} \bar{h}) \cdot {}_{c_{3}^{-1}c_{1}^{-1}}(c_2) \cdot c_4c_5c_5c_4c_3 \cdot {}_{c_{1}^{-1}}(c_2) $     \\

$\sim  {}_{c_{1}}(c_2) \cdot c_3c_4 \cdot c_5^2c_1^2 \cdot c_4 \cdot {}_{c_{3}c_{1}^{-1}}(c_2) \cdot c_5 \cdot {}_{c_{1}^{-1}}(\bar{k} \bar{h}) \cdot {}_{c_{3}^{-1}c_{1}^{-1}}(c_2) \cdot c_4c_5c_5c_4c_3 \cdot {}_{c_{1}^{-1}}(c_2) $     \\

$ \overset{L} {\rightarrow} \  {}_{c_{1}}(c_2) \cdot c_3c_4 \cdot \delta x c_3 \cdot c_4 \cdot {}_{c_{3}c_{1}^{-1}}(c_2) \cdot c_5 \cdot {}_{c_{1}^{-1}}(\bar{k} \bar{h}) \cdot {}_{c_{3}^{-1}c_{1}^{-1}}(c_2) \cdot c_4c_5c_5c_4c_3 \cdot {}_{c_{1}^{-1}}(c_2)=\textbf{Z(2)} $      \\

$ \sim  {}_{c_{1}}(c_2) \cdot c_3 c_4 \cdot \delta x \cdot {}_{c_{3}}(c_4) \cdot {}_{c_{3}^2 c_{1}^{-1}}(c_2) \cdot c_5 \cdot {}_{c_{1}^{-1}c_3}(\bar{k} \bar{h}) \cdot c_3 \cdot {}_{c_{3}^{-1}c_{1}^{-1}}(c_2) \cdot c_4c_5c_5c_4c_3 \cdot {}_{c_{1}^{-1}}(c_2) $       \\

$ \sim  {}_{c_{1}}(c_2) \cdot c_3 c_4 \cdot \delta x \cdot {}_{c_{3}}(c_4) \cdot {}_{c_{3}^2 c_{1}^{-1}}(c_2) \cdot c_5 \cdot {}_{c_{1}^{-1}c_3}(\bar{k} \bar{h}) \cdot {}_{c_{1}^{-1}}(c_2) \cdot c_3 \cdot c_4c_5c_5c_4c_3 \cdot {}_{c_{1}^{-1}}(c_2) $       \\

$ \sim  {}_{c_{1}}(c_2) \cdot c_3 c_4 \cdot \delta x \cdot {}_{c_{3}}(c_4) \cdot {}_{c_{3}^2 c_{1}^{-1}}(c_2) \cdot c_5 \cdot {}_{c_{1}^{-1}c_3}(\bar{k} \bar{h}) \cdot {}_{c_{1}^{-1}}(c_2) \cdot {}_{c_{3}}(c_4) \cdot c_3^2c_5^2 \cdot {}_{c_{3}^{-1}}(c_4) \cdot {}_{c_{1}^{-1}}(c_2) $      \\

$  \overset{L} {\rightarrow} \  {}_{c_{1}}(c_2) \cdot c_3 c_4 \cdot \delta x \cdot {}_{c_{3}}(c_4) \cdot {}_{c_{3}^2 c_{1}^{-1}}(c_2) \cdot c_5 \cdot {}_{c_{1}^{-1}c_3}(\bar{k} \bar{h}) \cdot {}_{c_{1}^{-1}}(c_2) \cdot {}_{c_{3}}(c_4) \cdot k h c_1 \cdot {}_{c_{3}^{-1}}(c_4) \cdot {}_{c_{1}^{-1}}(c_2)= \textbf{Z(3)} $    \\

$  \sim  {}_{c_{1}}(c_2) \cdot c_3 c_4 \cdot \delta x \cdot {}_{c_{3}}(c_4) \cdot {}_{c_{3}^2 c_{1}^{-1}}(c_2) \cdot c_5 \cdot {}_{c_{1}^{-1}c_3}(\bar{k} \bar{h}) \cdot {}_{c_{1}^{-1}}(c_2) \cdot {}_{c_{3}}(c_4) \cdot k h \cdot {}_{c_{3}^{-1}}(c_4) \cdot c_2 c_1 $    \\

$  \sim  {}_{c_{1}}(c_2) \cdot c_3 c_4 \cdot \delta x \cdot {}_{c_{3}}(c_4) \cdot {}_{c_{3}^2 c_{1}^{-1}}(c_2) \cdot c_5 \cdot {}_{c_{1}^{-1}c_3}(\bar{k} \bar{h}) \cdot {}_{c_{1}^{-1}}(c_2) \cdot c_2 \cdot {}_{c_{2}^{-1}c_{3}}(c_4) \cdot {}_{c_{2}^{-1}}(k h) \cdot {}_{c_{2}^{-1}c_{3}^{-1}}(c_4) \cdot c_1 $    \\

$  \overset{B} {\rightarrow} \  {}_{c_{1}}(c_2) \cdot c_3 c_4 \cdot \delta x \cdot {}_{c_{3}}(c_4) \cdot {}_{c_{3}^2 c_{1}^{-1}}(c_2) \cdot c_5 \cdot {}_{c_{1}^{-1}c_3}(\bar{k} \bar{h}) \cdot c_2 c_1 \cdot {}_{c_{2}^{-1}c_{3}}(c_4) \cdot {}_{c_{2}^{-1}}(k h) \cdot {}_{c_{2}^{-1}c_{3}^{-1}}(c_4) \cdot c_1 $    \\

$  \sim  {}_{c_{1}}(c_2) \cdot c_3 \cdot {}_{c_{4}}(\delta x) \cdot c_4 \cdot {}_{c_{3}}(c_4) \cdot {}_{c_{3}^2 c_{1}^{-1}}(c_2) \cdot c_5 \cdot {}_{c_{1}^{-1}c_3}(\bar{k} \bar{h}) \cdot c_2 c_1 \cdot {}_{c_{2}^{-1}c_{3}}(c_4) \cdot {}_{c_{2}^{-1}}(k h) \cdot {}_{c_{2}^{-1}c_{3}^{-1}}(c_4) \cdot c_1 $    \\

$  \overset{B} {\rightarrow} \  {}_{c_{1}}(c_2) \cdot c_3 \cdot {}_{c_{4}}(\delta x) \cdot c_3 c_4 \cdot {}_{c_{3}^2 c_{1}^{-1}}(c_2) \cdot c_5 \cdot {}_{c_{1}^{-1}c_3}(\bar{k} \bar{h}) \cdot c_2 c_1 \cdot {}_{c_{2}^{-1}c_{3}}(c_4) \cdot {}_{c_{2}^{-1}}(k h) \cdot {}_{c_{2}^{-1}c_{3}^{-1}}(c_4) \cdot c_1 $    \\

$  \sim  {}_{c_{1}}(c_2) \cdot c_3 \cdot {}_{c_{4}}(\delta x) \cdot c_3 c_4 \cdot {}_{c_{3}^2 c_{1}^{-1}}(c_2) \cdot c_5 \cdot {}_{c_{1}^{-1}c_3}(\bar{k} \bar{h}) \cdot c_2 \cdot c_1^2 \cdot {}_{c_{1}^{-1}c_{2}^{-1}c_{3}}(c_4) \cdot {}_{c_{1}^{-1}c_{2}^{-1}}(k h) \cdot {}_{c_{1}^{-1}c_{2}^{-1}c_{3}^{-1}}(c_4) $    \\

$  \sim  {}_{c_{1}}(c_2)  \cdot {}_{c_3 c_{4}}(\delta x) \cdot c_3^2 \cdot c_4 \cdot {}_{c_{3}^2 c_{1}^{-1}}(c_2) \cdot c_5 \cdot {}_{c_{1}^{-1}c_3}(\bar{k} \bar{h}) \cdot c_2 \cdot c_1^2 \cdot {}_{c_{1}^{-1}c_{2}^{-1}c_{3}}(c_4) \cdot {}_{c_{1}^{-1}c_{2}^{-1}}(k h) \cdot {}_{c_{1}^{-1}c_{2}^{-1}c_{3}^{-1}}(c_4) $    \\

$  \sim   {}_{c_{1}}(c_2)  \cdot {}_{c_3 c_{4}}(\delta x) \cdot {}_{c_3^2}(c_4) \cdot {}_{c_{3}^4 c_{1}^{-1}}(c_2) \cdot c_3^2 \cdot c_5 \cdot {}_{c_{1}^{-1}c_3}(\bar{k} \bar{h}) \cdot c_2 \cdot c_1^2 \cdot {}_{c_{1}^{-1}c_{2}^{-1}c_{3}}(c_4) \cdot {}_{c_{1}^{-1}c_{2}^{-1}}(k h) \cdot {}_{c_{1}^{-1}c_{2}^{-1}c_{3}^{-1}}(c_4) $    \\

$  \sim   {}_{c_{1}}(c_2)  \cdot {}_{c_3 c_{4}}(\delta x) \cdot {}_{c_3^2}(c_4) \cdot {}_{c_{3}^4 c_{1}^{-1}}(c_2) \cdot c_5 \cdot {}_{c_3^2 c_{1}^{-1} c_3}(\bar{k} \bar{h}) \cdot {}_{c_{3}^2}(c_2) \cdot c_1^2 c_3^2 \cdot {}_{c_{1}^{-1}c_{2}^{-1}c_{3}}(c_4) \cdot {}_{c_{1}^{-1}c_{2}^{-1}}(k h) \cdot {}_{c_{1}^{-1}c_{2}^{-1}c_{3}^{-1}}(c_4) $    \\

$  \overset{L} {\rightarrow} \ {}_{c_{1}}(c_2)  \cdot {}_{c_3 c_{4}}(\delta x) \cdot {}_{c_3^2}(c_4) \cdot {}_{c_{3}^4 c_{1}^{-1}}(c_2) \cdot c_5 \cdot {}_{c_3^2 c_{1}^{-1} c_3}(\bar{k} \bar{h}) \cdot {}_{c_{3}^2}(c_2) \cdot c_5 \bar{k} \bar{h} \cdot {}_{c_{1}^{-1}c_{2}^{-1}c_{3}}(c_4) \cdot {}_{c_{1}^{-1}c_{2}^{-1}}(k h) \cdot {}_{c_{1}^{-1}c_{2}^{-1}c_{3}^{-1}}(c_4) = \textbf{Z(4)}$    \\

\end{proof}

We collect positive relations of $Z(m)$ for $0 \leq m \leq 4$ below,  \\

\begin{itemize}

  \item $(c_1c_2c_3c_4{c_5}^2c_4c_3c_2c_1)^2 = \textbf{Z(0)} $    \\

  \item $c_1c_2c_3c_4c_5c_5c_4 \cdot {}_{c_{3}}(c_2) \cdot c_5 \bar{k} \bar{h} \cdot {}_{c_{3}^{-1}}(c_2) \cdot c_4c_5c_5c_4c_3c_2c_1=\textbf{Z(1)} $    \\
  
  \item ${}_{c_{1}}(c_2) \cdot c_3c_4 \cdot \delta x c_3 \cdot c_4 \cdot {}_{c_{3}c_{1}^{-1}}(c_2) \cdot c_5 \cdot {}_{c_{1}^{-1}}(\bar{k} \bar{h}) \cdot {}_{c_{3}^{-1}c_{1}^{-1}}(c_2) \cdot c_4c_5c_5c_4c_3 \cdot {}_{c_{1}^{-1}}(c_2)=\textbf{Z(2)} $    \\
  
  \item ${}_{c_{1}}(c_2) \cdot c_3 c_4 \cdot \delta x \cdot {}_{c_{3}}(c_4) \cdot {}_{c_{3}^2 c_{1}^{-1}}(c_2) \cdot c_5 \cdot {}_{c_{1}^{-1}c_3}(\bar{k} \bar{h}) \cdot {}_{c_{1}^{-1}}(c_2) \cdot {}_{c_{3}}(c_4) \cdot k h c_1 \cdot {}_{c_{3}^{-1}}(c_4) \cdot {}_{c_{1}^{-1}}(c_2)= \textbf{Z(3)} $    \\
  
  \item $  {}_{c_{1}}(c_2)  \cdot {}_{c_3 c_{4}}(\delta x) \cdot {}_{c_3^2}(c_4) \cdot {}_{c_{3}^4 c_{1}^{-1}}(c_2) \cdot c_5 \cdot {}_{c_3^2 c_{1}^{-1} c_3}(\bar{k} \bar{h}) \cdot {}_{c_{3}^2}(c_2) \cdot c_5 \bar{k} \bar{h} \cdot {}_{c_{1}^{-1}c_{2}^{-1}c_{3}}(c_4) \cdot {}_{c_{1}^{-1}c_{2}^{-1}}(k h) \cdot {}_{c_{1}^{-1}c_{2}^{-1}c_{3}^{-1}}(c_4) = \textbf{Z(4)}$    \\
  
\end{itemize}

\par We see that $Z(m)$ for $1 \leq m \leq 4$ are rational blowdown copies of $\CP \# 13\CPb$ by the Theorem 3.1 in~\cite{EG} constructed in the same methodology as in the Endo-Gurtas and thus homeomorphic to $\CP\# (13-m)\CPb$ for $0 \leq m \leq 4$ (cf. \cite{EG}) except that we have avoided in using the conjugation move $ \overset{C} {\rightarrow}$ which will facilitate fiber sum splitting of $X(n)$ for $2 \leq n \leq 6$ constructed below. It is also worth noting that $Z(1)$ is diffeomorphic to $\CP \# 12\CPb$ by Proposition~\ref{181}. After this we can characterize the $Z(2), Z(3), Z(4)$ only up to homeomorphism types of $\CP \# (12-m)\CPb$ for $2 \leq m \leq 4$.

\par Combining the above monodromy computations for $X(0)$ and $Z(m)$. We can now find six lantern relations on $K3\#2\CPb$ which allows the fiber sum decomposability.

\begin{lem}[Monodromy of $X(n)$ for $2 \leq n \leq 6$] \label{sixl}
The global monodromy of genus 2 Lefschetz fibration on $K3\#2\CPb$ over $S^{2}$ given by the relation $X(0)$ in Proposition~\ref{twol} can be conjugated and braid substituted to contain six lantern relations.
\end{lem}

\begin{proof}

We begin with $\textbf{X(0)} = (B_0 B_1 \cdot c_1c_2c_3c_4{c_5}^2 c_4 \cdot {}_{c_{3}}(c_2) \cdot c_1 \cdot c_1^2 c_5^2)^2 = 1 $ \\

$ \sim B_0 B_1 \cdot c_1c_2c_3c_4{c_5}^2 c_4 \cdot {}_{c_{3}}(c_2) \cdot c_1 \cdot c_1^2 c_5^2 \cdot B_0 B_1 \cdot c_1c_2c_3c_4{c_5}^2 c_4 \cdot {}_{c_{3}}(c_2) \cdot c_1 \cdot c_1^2 c_5^2 $ \\

$ \overset{L} {\rightarrow} \ B_0 B_1 \cdot c_1c_2c_3c_4{c_5}^2 c_4 \cdot {}_{c_{3}}(c_2) \cdot c_1 \cdot c_1^2 c_5^2 \cdot B_0 B_1 \cdot c_1c_2c_3c_4{c_5}^2 c_4 \cdot {}_{c_{3}}(c_2) \cdot c_1 \cdot x c_3 \delta  = \textbf{X(1)}$ \\

$ \overset{L} {\rightarrow} \ B_0 B_1 \cdot c_1c_2c_3c_4{c_5}^2 c_4 \cdot {}_{c_{3}}(c_2) \cdot c_1 \cdot x c_3 \delta \cdot B_0 B_1 \cdot c_1c_2c_3c_4{c_5}^2 c_4 \cdot {}_{c_{3}}(c_2) \cdot c_1 \cdot x c_3 \delta = \textbf{X(2)}$ \\

$ \sim B_0 B_1 \cdot c_1c_2c_3c_4{c_5}^2 c_4 \cdot {}_{c_{3}}(c_2) \cdot c_1 \cdot c_3 \cdot {}_{c_{3}^{-1}}(x) \cdot \delta B_0 B_1 \cdot c_1c_2c_3c_4{c_5}^2 c_4 \cdot {}_{c_{3}}(c_2) \cdot c_1 \cdot c_3 \cdot {}_{c_{3}^{-1}}(x) \cdot \delta  $ \\

$ \sim B_0 B_1 \cdot (c_1c_2c_3c_4{c_5}^2c_4c_3c_2c_1) \cdot {}_{c_{3}^{-1}}(x) \cdot \delta B_0 B_1 \cdot (c_1c_2c_3c_4{c_5}^2c_4c_3c_2c_1) \cdot {}_{c_{3}^{-1}}(x) \cdot \delta $\\ 

$ \sim B_0 B_1 \cdot {}_{c_{3}^{-1}}(x) \cdot \delta B_0 B_1 \cdot {}_{c_{3}^{-1}}(x) \cdot \delta \cdot c_1c_2c_3c_4{c_5}^2c_4c_3c_2c_1 \cdot c_1c_2c_3c_4{c_5}^2c_4c_3c_2c_1 \\ $
\{$(c_1c_2c_3c_4{c_5}^2c_4c_3c_2c_1)$ is central\}\\

$ \sim B_0 B_1 B_2 \delta B_0 B_1 B_2 \delta \cdot c_1c_2c_3c_4{c_5}^2c_4c_3c_2c_1 \cdot c_1c_2c_3c_4{c_5}^2c_4c_3c_2c_1 $\\ 
\{$B_2 := {}_{c_{3}^{-1}}(x)\}$\\

$ \sim (B_0 B_1 B_2 \delta)^2 \cdot (c_1c_2c_3c_4{c_5}^2c_4c_3c_2c_1)^2 = \textbf{X(2)}$  \\


$ \sim (B_0 B_1 B_2 \delta)^2 \cdot c_1c_2c_3c_4c_5c_5c_4c_3c_2c_1c_1c_2c_3c_4c_5c_5c_4c_3c_2c_1 $\\

$ \sim (B_0 B_1 B_2 \delta)^2 \cdot c_1c_2c_3c_4c_5c_5c_4 \cdot {}_{c_{3}}(c_2) \cdot c_1^2c_3^2 \cdot {}_{c_{3}^{-1}}(c_2) \cdot c_4c_5c_5c_4c_3c_2c_1 $\\

$ \overset{L} {\rightarrow} \  (B_0 B_1 B_2 \delta)^2 \cdot c_1c_2c_3c_4c_5c_5c_4 \cdot {}_{c_{3}}(c_2) \cdot c_5 \bar{k} \bar{h} \cdot {}_{c_{3}^{-1}}(c_2) \cdot c_4c_5c_5c_4c_3c_2c_1 = \textbf{X(3)}$\\

$\sim  (B_0 B_1 B_2 \delta)^2 \cdot c_1c_2c_3c_4c_5c_5  c_4 \cdot c_1 \cdot {}_{c_{3}c_{1}^{-1}}(c_2) \cdot c_5 \cdot {}_{c_{1}^{-1}}(\bar{k} \bar{h}) \cdot {}_{c_{3}^{-1}c_{1}^{-1}}(c_2) \cdot c_4c_5c_5c_4c_3 \cdot {}_{c_{1}^{-1}}(c_2) $   \\

$\sim  (B_0 B_1 B_2 \delta)^2 \cdot {}_{c_{1}}(c_2) \cdot c_3c_4 \cdot c_5^2c_1^2 \cdot c_4 \cdot {}_{c_{3}c_{1}^{-1}}(c_2) \cdot c_5 \cdot {}_{c_{1}^{-1}}(\bar{k} \bar{h}) \cdot {}_{c_{3}^{-1}c_{1}^{-1}}(c_2) \cdot c_4c_5c_5c_4c_3 \cdot {}_{c_{1}^{-1}}(c_2) $   \\

$ \overset{L} {\rightarrow} \  (B_0 B_1 B_2 \delta)^2 \cdot {}_{c_{1}}(c_2) \cdot c_3c_4 \cdot \delta x c_3 \cdot c_4 \cdot {}_{c_{3}c_{1}^{-1}}(c_2) \cdot c_5 \cdot {}_{c_{1}^{-1}}(\bar{k} \bar{h}) \cdot {}_{c_{3}^{-1}c_{1}^{-1}}(c_2) \cdot c_4c_5c_5c_4c_3 \cdot {}_{c_{1}^{-1}}(c_2) = \textbf{X(4)}$   \\

$ \sim  (B_0 B_1 B_2 \delta)^2 \cdot {}_{c_{1}}(c_2) \cdot c_3 c_4 \cdot \delta x \cdot {}_{c_{3}}(c_4) \cdot {}_{c_{3}^2 c_{1}^{-1}}(c_2) \cdot c_5 \cdot {}_{c_{1}^{-1}c_3}(\bar{k} \bar{h}) \cdot c_3 \cdot {}_{c_{3}^{-1}c_{1}^{-1}}(c_2) \cdot c_4c_5c_5c_4c_3 \cdot {}_{c_{1}^{-1}}(c_2) $    \\

$ \sim  (B_0 B_1 B_2 \delta)^2 \cdot {}_{c_{1}}(c_2) \cdot c_3 c_4 \cdot \delta x \cdot {}_{c_{3}}(c_4) \cdot {}_{c_{3}^2 c_{1}^{-1}}(c_2) \cdot c_5 \cdot {}_{c_{1}^{-1}c_3}(\bar{k} \bar{h}) \cdot {}_{c_{1}^{-1}}(c_2) \cdot c_3 \cdot c_4c_5c_5c_4c_3 \cdot {}_{c_{1}^{-1}}(c_2) $    \\

$ \sim  (B_0 B_1 B_2 \delta)^2 \cdot {}_{c_{1}}(c_2) \cdot c_3 c_4 \cdot \delta x \cdot {}_{c_{3}}(c_4) \cdot {}_{c_{3}^2 c_{1}^{-1}}(c_2) \cdot c_5 \cdot {}_{c_{1}^{-1}c_3}(\bar{k} \bar{h}) \cdot {}_{c_{1}^{-1}}(c_2) \cdot {}_{c_{3}}(c_4) \cdot c_3^2c_5^2 \cdot {}_{c_{3}^{-1}}(c_4) \cdot {}_{c_{1}^{-1}}(c_2) $    \\

$  \overset{L} {\rightarrow} \  (B_0 B_1 B_2 \delta)^2 \cdot {}_{c_{1}}(c_2) \cdot c_3 c_4 \cdot \delta x \cdot {}_{c_{3}}(c_4) \cdot {}_{c_{3}^2 c_{1}^{-1}}(c_2) \cdot c_5 \cdot {}_{c_{1}^{-1}c_3}(\bar{k} \bar{h}) \cdot {}_{c_{1}^{-1}}(c_2) \cdot {}_{c_{3}}(c_4) \cdot k h c_1 \cdot {}_{c_{3}^{-1}}(c_4) \cdot {}_{c_{1}^{-1}}(c_2) = \textbf{X(5)}$    \\

$  \sim  (B_0 B_1 B_2 \delta)^2 \cdot {}_{c_{1}}(c_2) \cdot c_3 c_4 \cdot \delta x \cdot {}_{c_{3}}(c_4) \cdot {}_{c_{3}^2 c_{1}^{-1}}(c_2) \cdot c_5 \cdot {}_{c_{1}^{-1}c_3}(\bar{k} \bar{h}) \cdot {}_{c_{1}^{-1}}(c_2) \cdot {}_{c_{3}}(c_4) \cdot k h \cdot {}_{c_{3}^{-1}}(c_4) \cdot c_2 c_1 $    \\

$  \sim  (B_0 B_1 B_2 \delta)^2 \cdot {}_{c_{1}}(c_2) \cdot c_3 c_4 \cdot \delta x \cdot {}_{c_{3}}(c_4) \cdot {}_{c_{3}^2 c_{1}^{-1}}(c_2) \cdot c_5 \cdot {}_{c_{1}^{-1}c_3}(\bar{k} \bar{h}) \cdot {}_{c_{1}^{-1}}(c_2) \cdot c_2 \cdot {}_{c_{2}^{-1}c_{3}}(c_4) \cdot {}_{c_{2}^{-1}}(k h) \cdot {}_{c_{2}^{-1}c_{3}^{-1}}(c_4) \cdot c_1 $    \\

$  \overset{B} {\rightarrow} \  (B_0 B_1 B_2 \delta)^2 \cdot {}_{c_{1}}(c_2) \cdot c_3 c_4 \cdot \delta x \cdot {}_{c_{3}}(c_4) \cdot {}_{c_{3}^2 c_{1}^{-1}}(c_2) \cdot c_5 \cdot {}_{c_{1}^{-1}c_3}(\bar{k} \bar{h}) \cdot c_2 c_1 \cdot {}_{c_{2}^{-1}c_{3}}(c_4) \cdot {}_{c_{2}^{-1}}(k h) \cdot {}_{c_{2}^{-1}c_{3}^{-1}}(c_4) \cdot c_1 $    \\

$  \sim  (B_0 B_1 B_2 \delta)^2 \cdot {}_{c_{1}}(c_2) \cdot c_3 \cdot {}_{c_{4}}(\delta x) \cdot c_4 \cdot {}_{c_{3}}(c_4) \cdot {}_{c_{3}^2 c_{1}^{-1}}(c_2) \cdot c_5 \cdot {}_{c_{1}^{-1}c_3}(\bar{k} \bar{h}) \cdot c_2 c_1 \cdot {}_{c_{2}^{-1}c_{3}}(c_4) \cdot {}_{c_{2}^{-1}}(k h) \cdot {}_{c_{2}^{-1}c_{3}^{-1}}(c_4) \cdot c_1 $    \\

$  \overset{B} {\rightarrow} \  (B_0 B_1 B_2 \delta)^2 \cdot {}_{c_{1}}(c_2) \cdot c_3 \cdot {}_{c_{4}}(\delta x) \cdot c_3 c_4 \cdot {}_{c_{3}^2 c_{1}^{-1}}(c_2) \cdot c_5 \cdot {}_{c_{1}^{-1}c_3}(\bar{k} \bar{h}) \cdot c_2 c_1 \cdot {}_{c_{2}^{-1}c_{3}}(c_4) \cdot {}_{c_{2}^{-1}}(k h) \cdot {}_{c_{2}^{-1}c_{3}^{-1}}(c_4) \cdot c_1 $    \\

$  \sim  (B_0 B_1 B_2 \delta)^2 \cdot {}_{c_{1}}(c_2) \cdot c_3 \cdot {}_{c_{4}}(\delta x) \cdot c_3 c_4 \cdot {}_{c_{3}^2 c_{1}^{-1}}(c_2) \cdot c_5 \cdot {}_{c_{1}^{-1}c_3}(\bar{k} \bar{h}) \cdot c_2 \cdot c_1^2 \cdot {}_{c_{1}^{-1}c_{2}^{-1}c_{3}}(c_4) \cdot {}_{c_{1}^{-1}c_{2}^{-1}}(k h) \cdot {}_{c_{1}^{-1}c_{2}^{-1}c_{3}^{-1}}(c_4) $    \\

$  \sim  (B_0 B_1 B_2 \delta)^2 \cdot {}_{c_{1}}(c_2)  \cdot {}_{c_3 c_{4}}(\delta x) \cdot c_3^2 \cdot c_4 \cdot {}_{c_{3}^2 c_{1}^{-1}}(c_2) \cdot c_5 \cdot {}_{c_{1}^{-1}c_3}(\bar{k} \bar{h}) \cdot c_2 \cdot c_1^2 \cdot {}_{c_{1}^{-1}c_{2}^{-1}c_{3}}(c_4) \cdot {}_{c_{1}^{-1}c_{2}^{-1}}(k h) \cdot {}_{c_{1}^{-1}c_{2}^{-1}c_{3}^{-1}}(c_4) $    \\

$  \sim  (B_0 B_1 B_2 \delta)^2 \cdot {}_{c_{1}}(c_2)  \cdot {}_{c_3 c_{4}}(\delta x) \cdot {}_{c_3^2}(c_4) \cdot {}_{c_{3}^4 c_{1}^{-1}}(c_2) \cdot c_3^2 \cdot c_5 \cdot {}_{c_{1}^{-1}c_3}(\bar{k} \bar{h}) \cdot c_2 \cdot c_1^2 \cdot {}_{c_{1}^{-1}c_{2}^{-1}c_{3}}(c_4) \cdot {}_{c_{1}^{-1}c_{2}^{-1}}(k h) \cdot {}_{c_{1}^{-1}c_{2}^{-1}c_{3}^{-1}}(c_4) $    \\

$  \sim  (B_0 B_1 B_2 \delta)^2 \cdot {}_{c_{1}}(c_2)  \cdot {}_{c_3 c_{4}}(\delta x) \cdot {}_{c_3^2}(c_4) \cdot {}_{c_{3}^4 c_{1}^{-1}}(c_2) \cdot c_5 \cdot {}_{c_3^2 c_{1}^{-1} c_3}(\bar{k} \bar{h}) \cdot {}_{c_{3}^2}(c_2) \cdot c_1^2 c_3^2 \cdot {}_{c_{1}^{-1}c_{2}^{-1}c_{3}}(c_4) \cdot {}_{c_{1}^{-1}c_{2}^{-1}}(k h) \cdot {}_{c_{1}^{-1}c_{2}^{-1}c_{3}^{-1}}(c_4) $    \\

$  \overset{L} {\rightarrow} \  (B_0 B_1 B_2 \delta)^2 \cdot {}_{c_{1}}(c_2)  \cdot {}_{c_3 c_{4}}(\delta x) \cdot {}_{c_3^2}(c_4) \cdot {}_{c_{3}^4 c_{1}^{-1}}(c_2) \cdot c_5 \cdot {}_{c_3^2 c_{1}^{-1} c_3}(\bar{k} \bar{h}) \cdot {}_{c_{3}^2}(c_2) \cdot c_5 \bar{k} \bar{h} \cdot {}_{c_{1}^{-1}c_{2}^{-1}c_{3}}(c_4) \cdot {}_{c_{1}^{-1}c_{2}^{-1}}(k h) \cdot {}_{c_{1}^{-1}c_{2}^{-1}c_{3}^{-1}}(c_4) = \textbf{X(6)}$    \\

\end{proof}

We collect positive relations of $X(n)$ for $2 \leq n \leq 6$ below,  \\

\begin{itemize}
  \item $(B_0 B_1 B_2 \delta)^2 \cdot (c_1c_2c_3c_4{c_5}^2c_4c_3c_2c_1)^2 = \textbf{X(2)}$    \\
  \item $(B_0 B_1 B_2 \delta)^2 \cdot c_1c_2c_3c_4c_5c_5c_4 \cdot {}_{c_{3}}(c_2) \cdot c_5 \bar{k} \bar{h} \cdot {}_{c_{3}^{-1}}(c_2) \cdot c_4c_5c_5c_4c_3c_2c_1 = \textbf{X(3)}$    \\
  \item $(B_0 B_1 B_2 \delta)^2 \cdot {}_{c_{1}}(c_2) \cdot c_3c_4 \cdot \delta x c_3 \cdot c_4 \cdot {}_{c_{3}c_{1}^{-1}}(c_2) \cdot c_5 \cdot {}_{c_{1}^{-1}}(\bar{k} \bar{h}) \cdot {}_{c_{3}^{-1}c_{1}^{-1}}(c_2) \cdot c_4c_5c_5c_4c_3 \cdot {}_{c_{1}^{-1}}(c_2) = \textbf{X(4)}$    \\
  \item $(B_0 B_1 B_2 \delta)^2 \cdot {}_{c_{1}}(c_2) \cdot c_3 c_4 \cdot \delta x \cdot {}_{c_{3}}(c_4) \cdot {}_{c_{3}^2 c_{1}^{-1}}(c_2) \cdot c_5 \cdot {}_{c_{1}^{-1}c_3}(\bar{k} \bar{h}) \cdot {}_{c_{1}^{-1}}(c_2) \cdot {}_{c_{3}}(c_4) \cdot k h c_1 \cdot {}_{c_{3}^{-1}}(c_4) \cdot {}_{c_{1}^{-1}}(c_2) = \textbf{X(5)}$    \\
  \item $(B_0 B_1 B_2 \delta)^2 \cdot {}_{c_{1}}(c_2)  \cdot {}_{c_3 c_{4}}(\delta x) \cdot {}_{c_3^2}(c_4) \cdot {}_{c_{3}^4 c_{1}^{-1}}(c_2) \cdot c_5 \cdot {}_{c_3^2 c_{1}^{-1} c_3}(\bar{k} \bar{h}) \cdot {}_{c_{3}^2}(c_2) \cdot c_5 \bar{k} \bar{h} \cdot {}_{c_{1}^{-1}c_{2}^{-1}c_{3}}(c_4) \cdot {}_{c_{1}^{-1}c_{2}^{-1}}(k h) \cdot {}_{c_{1}^{-1}c_{2}^{-1}c_{3}^{-1}}(c_4) = \textbf{X(6)}$   \\
\end{itemize}

We now give a proof of the main theorem of this section Theorem~\ref{theorem1}. 

\begin{proof} 
Let $X(n)$ for $2 \leq n \leq 6$ be the symplectic $4$-manifold obtained from $K3\#2\CPb$ by applying a sequence of six lantern relation substitutions as in Lemma~\ref{sixl} above. 

We first compute the topological invariants to determine the homeomorphism types of $X(n)$ for $2 \leq n \leq 6$.
 
 \begin{eqnarray*}
e(X(n)) &=& \# {\emph{singular fibers}} + e(\mathbb{S}^{2}) e(\Sigma_{2}) =  (30-n) + 2(-2) =  26 - n ,\\
\sigma(X(n)) &=& -\frac{3}{5}n - \frac{1}{5}s =  -\frac{3}{5}(30 - 2n) - \frac{1}{5}(n) = -18 + n, \\
{c_1}^{2}(X(n)) &:=& 3\sigma(X(n)) + 2e(X(n)) = n - 2, \\
\chi(X(n)) &:=& (e(X(n)) + \sigma(X(n)))/4 = 2 \\
\end{eqnarray*}

\par $X(n)$ for $2 \leq n \leq 6$ are simply-connected as $X(0) = K3\#2\,\CPb$ and $X(n-1) \setminus C_2$ are simply-connected. They have the Euler characteristic $e(X(n)) = e(K3\#2\CPb) - n = (26) - n$ with the signature $\sigma(X(n)) = \sigma (K3\#2\CPb) + n = (-18) + n$. Note that they all are non-spin as there are reducible fibers \cite{Stipsicz2}. All together, $X(n)$ for $2 \leq n \leq 6$ are homeomorphic to $3\CP\# (21-n)\CPb$ from Freedman's classification theorem (cf.\ \cite{Freedman}).

Using the blow up formula for the Seiberg-Witten function \cite{FS2}, we have $SW_{K3\#2\,\CPb}$ 
$= SW_{K3} \cdot \prod_{j=1}^{2}(e^{E_{i}} + e^{-E_{i}}) = (e^{E_{1}} + e^{-E_{1}})(e^{E_{2}} + e^{-E_{2}})$, where $E_{i}$ is an exceptional class coming from the $i^{th}$ blow up. Thus the set of basic classes of $K3\#2\,\CPb$ are given by $\pm E_{1} \pm E_{2}$, and the Seiberg-Witten invariants on these classes are $\pm 1$. After performing one rational blowdown along a copy of the configuration $C_2$, the resulting manifold is diffeomorphic to $K3\#\,\CPb$ by Lemma~\ref{n=2}. Thus, the only basic classes are $\pm E$, where $E \in H^2(K3\#\CPb)$ is the poincar\'e dual of the homology class of the exceptional sphere, which decends from the top classes ${\pm (E_{1} + E_{2})}$ in $K3\#2\,\CPb$. Next, using the Corollary 8.6 in~\cite{FS1}, we see that $X$ has Seiberg-Witten simple type. By applying Theorem~\ref{SW1} and Theorem~\ref{SW2}, we completely determine the Seiberg-Witten invariants of $X$ using the basic classes and invariants of $K3\#\,\CPb$: Up to sign the symplectic manifold $X$ has only one basic class which descends from the canonical class of $K3\#\,\CPb$. By Theorem ~\ref{SW2} (or by Taubes theorem \cite{Taubes1}), the value of the Seiberg-Witten function on these classes, $\pm K_{X(n)}$, are ${ \pm 1}$.           

By using Fintushel-Stern's rational blowdown formula we can also determine the Seiberg-Witten invariants of $X(n)$ for $2 \leq n \leq 6$ directly by computing the algebraic intersection number of the classes $\pm E_{1} \pm E_{2}$ with the classes of $-4$ spheres of six $C_2$ configurations. Note that these $-4$ spheres are the components of the singular fibers of $K3\#2\,\CPb$. As three regions on the genus two surface, where the rational blowdowns are performed always intersect the two exceptional divisors once (cf. \cite{AP}), we compute the intersection numbers as follows: Let $S$ denote the homology class of $-4$ sphere of $C_{2}$. We have $S \cdot E_{1} = S \cdot E_{2} = 1$. Consequently, $S \cdot \pm (E_{1} + E_{2}) = \pm 2$ and $S \cdot \pm (E_{1} - E_{2}) = 0$. Since among the four classes $\pm E_{1} \pm E_{2}$ only $E_{1} + E_{2}$ and $-(E_{1} + E_{2})$ have intersection $\pm 2$ with $-4$ spheres of $C_2$, it follows from Theorem~\ref{SW1} that these are only two classes that descend to $X(n)$ for $2 \leq n \leq 6$.

Next, we apply the connected sum theorem for the Seiberg-Witten invariant and show that $SW$ function is trivial for $3\CP\# (21-n)\CPb$ for $2 \leq n \leq 6$. Since the Seiberg-Witten invariants are diffeomorphism invariants, we conclude that $X(n)$ for $2 \leq n \leq 6$ are not diffeomorphic to $3\CP\# (21-n)\CPb$ for $2 \leq n \leq 6$. 

Using the Seiberg-Witten basic classes, the minimality of $X(n)$ for $2 \leq n \leq 6$ follows from the the fact that $X(n)$ for $2 \leq n \leq 6$ has no two basic classes $K$ and $K'$ such that $(K - K')^2 = -4$. Notice that $(K_{X(n)} - (-K_{X(n)}))^2 = 4({K_X(n)}^{2}) = 16$ for $2 \leq n \leq 6$ in our case.

The symplectic Kodaira dimension $\kappa^s(X(n))$ for $2 \leq n \leq 6$ are equal to $\kappa^s = 1$ for $n=2$ and $\kappa^s = 2$ for $3 \leq n \leq 6$. The $X(2)$ has $\kappa^s(X(2))=1$ since it is a minimal exotic copy of $3\CP\# 19\CPb$ (cf. \cite{DZ, Li}). Finally, $\kappa^s(X(n))=2$ for $3 \leq n \leq 6$ since they are also minimal and have $c_1^2(X(n)) \geq 0$.

Thus the $X(n)$ for $1 \leq n \leq 6$ are simply-connected, symplectic $4$-manifolds homeomorphic but not diffeomorphic to $3\CP\# (21-n)\CPb$ with $b_2^+=3$ and symplectically minimal for $2 \leq n \leq 6$ with symplectic Kodaira dimension $\kappa^s = 0$ for $n=1$, $\kappa^s = 1$ for $n=2$ and $\kappa^s = 2$ for $3 \leq n \leq 6$.

\end{proof}

\section{Construction of $X(7)$}

In this section, we will find one more lantern relation from positive relation of $X(6)$ by replacing the Matsumoto's fibration summand $\mathbb{S}^2 \times \mathbb{T}^2 \# 4 \CPb$ with its globally conjugated copy having the positive relation of

\begin{align}
(c_{1})^{2}(Y_1 Y_2 Y_c)^{2} = 1
\end{align} 

which is introduced in Proposition ~\ref{m2}.

With this relation we can replace the Matsumoto's fibration summand word $(B_0 B_1 B_2 \delta)^2$ in $X(6)$ with $(c_{1})^{2}(Y_1 Y_2 Y_c)^{2}$ which would allow us to perform the seventh lantern substitution.

We can now perform one more lantern substitution via the following Hurwitz move.

\begin{proof}

We begin with relation of $X(6)$

$\textbf{X(6)} = (B_0 B_1 B_2 \delta)^2 \cdot {}_{c_{1}}(c_2)  \cdot {}_{c_3 c_{4}}(\delta x) \cdot {}_{c_3^2}(c_4) \cdot {}_{c_{3}^4 c_{1}^{-1}}(c_2) \cdot c_5 \cdot {}_{c_3^2 c_{1}^{-1} c_3}(\bar{k} \bar{h}) \cdot {}_{c_{3}^2}(c_2) \cdot c_5 \bar{k} \bar{h} \cdot {}_{c_{1}^{-1}c_{2}^{-1}c_{3}}(c_4) \cdot {}_{c_{1}^{-1}c_{2}^{-1}}(k h) \cdot {}_{c_{1}^{-1}c_{2}^{-1}c_{3}^{-1}}(c_4) = 1$   \\

$ \sim (c_{1})^{2}(Y_1 Y_2 Y_c)^{2} \cdot {}_{c_{1}}(c_2)  \cdot {}_{c_3 c_{4}}(\delta x) \cdot {}_{c_3^2}(c_4) \cdot {}_{c_{3}^4 c_{1}^{-1}}(c_2) \cdot c_5 \cdot {}_{c_3^2 c_{1}^{-1} c_3}(\bar{k} \bar{h}) \cdot {}_{c_{3}^2}(c_2) \cdot c_5 \bar{k} \bar{h} \cdot {}_{c_{1}^{-1}c_{2}^{-1}c_{3}}(c_4) \cdot {}_{c_{1}^{-1}c_{2}^{-1}}(k h) \cdot {}_{c_{1}^{-1}c_{2}^{-1}c_{3}^{-1}}(c_4)$   \\

$ \overset{C} {\rightarrow} \ (Y_1 Y_2 Y_c)^{2} \cdot {}_{c_{1}}(c_2)  \cdot {}_{c_3 c_{4}}(\delta x) \cdot {}_{c_3^2}(c_4) \cdot {}_{c_{3}^4 c_{1}^{-1}}(c_2) \cdot c_5 \cdot {}_{c_3^2 c_{1}^{-1} c_3}(\bar{k} \bar{h}) \cdot {}_{c_{3}^2}(c_2) \cdot c_5 \bar{k} \bar{h} \cdot {}_{c_{1}^{-1}c_{2}^{-1}c_{3}}(c_4) \cdot {}_{c_{1}^{-1}c_{2}^{-1}}(k h) \cdot {}_{c_{1}^{-1}c_{2}^{-1}c_{3}^{-1}}(c_4) \cdot c_{1}^{2}$   \\

$ \sim (Y_1 Y_2 Y_c)^{2} \cdot {}_{c_{1}}(c_2)  \cdot {}_{c_3 c_{4}}(\delta x) \cdot {}_{c_3^2}(c_4) \cdot {}_{c_{3}^4 c_{1}^{-1}}(c_2) \cdot {}_{c_5c_3^2 c_{1}^{-1} c_3}(\bar{k} \bar{h}) \cdot {}_{c_5c_{3}^2}(c_2) \cdot c_{5}^2 \cdot \bar{k} \bar{h} \cdot {}_{c_{1}^{-1}c_{2}^{-1}c_{3}}(c_4) \cdot {}_{c_{1}^{-1}c_{2}^{-1}}(k h) \cdot {}_{c_{1}^{-1}c_{2}^{-1}c_{3}^{-1}}(c_4)\cdot c_{1}^{2}$   \\

$ \sim (Y_1 Y_2 Y_c)^{2} \cdot {}_{c_{1}}(c_2)  \cdot {}_{c_3 c_{4}}(\delta x) \cdot {}_{c_3^2}(c_4) \cdot {}_{c_{3}^4 c_{1}^{-1}}(c_2) \cdot {}_{c_5c_3^2 c_{1}^{-1} c_3}(\bar{k} \bar{h}) \cdot {}_{c_5c_{3}^2}(c_2) \cdot  {}_{c_{5}^{2}}(\bar{k} \bar{h}) \cdot {}_{c_{5}^{2}c_{1}^{-1}c_{2}^{-1}c_{3}}(c_4) \cdot {}_{c_{5}^{2}c_{1}^{-1}c_{2}^{-1}}(k h) \cdot {}_{c_{5}^{2}c_{1}^{-1}c_{2}^{-1}c_{3}^{-1}}(c_4)\cdot c_{5}^{2}c_{1}^{2}$   \\

$ \overset{L} {\rightarrow} \ (Y_1 Y_2 Y_c)^{2} \cdot {}_{c_{1}}(c_2)  \cdot {}_{c_3 c_{4}}(\delta x) \cdot {}_{c_3^2}(c_4) \cdot {}_{c_{3}^4 c_{1}^{-1}}(c_2) \cdot {}_{c_5c_3^2 c_{1}^{-1} c_3}(\bar{k} \bar{h}) \cdot {}_{c_5c_{3}^2}(c_2) \cdot  {}_{c_{5}^{2}}(\bar{k} \bar{h}) \cdot {}_{c_{5}^{2}c_{1}^{-1}c_{2}^{-1}c_{3}}(c_4) \cdot {}_{c_{5}^{2}c_{1}^{-1}c_{2}^{-1}}(k h) \cdot {}_{c_{5}^{2}c_{1}^{-1}c_{2}^{-1}c_{3}^{-1}}(c_4) \cdot c_{3} \delta x = \textbf{X(7)}$   \\

\end{proof}

We note that the total space of the $X(7)$ is a symplectically minimal exotic copy of $3\CP\# 14\CPb$ with $b_2^+=3$ and $c_1^2(X(7)) = 5$ by the similar argument as above.

\rmk The observation above that one more lantern relation could be found which allows a sequence of seven rational blowdowns to be performed on $K3\#2\,\CPb$ to acquire $X(7)$ rather six rational blowdowns could potentially have a deeper geometric meaning rather than merely constructing a smaller exotic 4-manifold. By the work of E. Hironaka \cite{Hironaka}, one can `read' latern relation from the planar line arrangement (Thm 1.2 in \cite{Hironaka}). In our case the lantern relation corresponds to the triangle formed by the lines in the branch locus of the double branched covering description for the $K3\#2\,\CPb$. Interestingly, there are correspondingly seven triangles in generic arrangement of six lines which is the linear system $|6\tilde{L}|$ for our branch locus. It is tempting to postulate that one can find seven lantern relation in the global monodromy of $K3\#2\,\CPb$ and we have found them by mapping class group factorization calculus. The concrete interplay between the change in the branch locus of the Hyperelliptic Lefschetz fibration and its braid group monodromy in connection with the change in the topological structure of the Hyperelliptic Lefschetz fibraiton and its mapping class group monodromy is an interesting avenue to be investigated upon which was surveyed by I. Smith and D. Auroux in \cite{AS}. We will continue our investigation on this topic in an upcoming project \cite{Park}.

\begin{figure}[ht]
\begin{center}
\includegraphics[scale=.3]{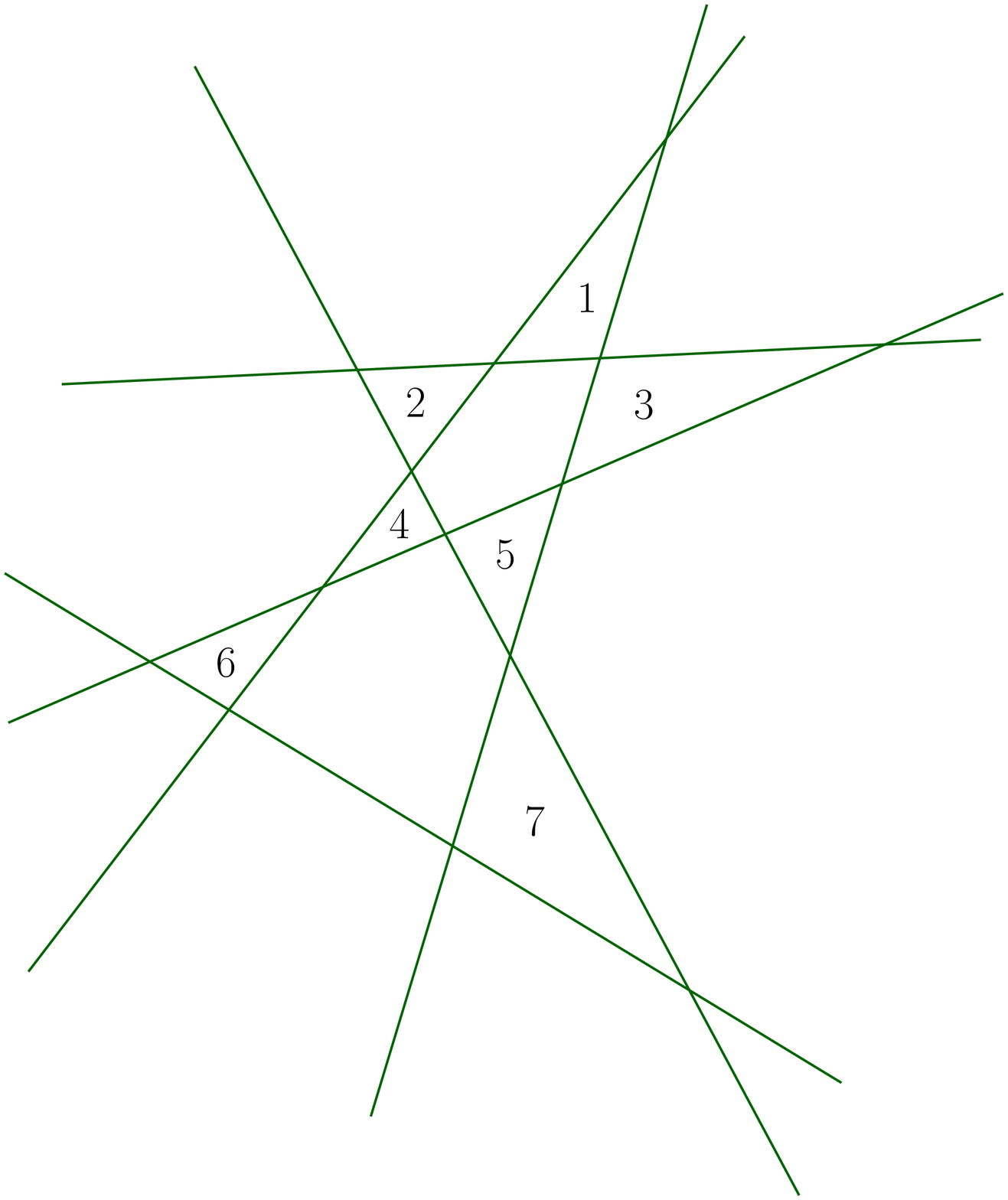}
\caption{Seven Triangles on Branch Locus for $K3\#2\CPb$}
\label{fig:branch}
\end{center}
\end{figure}

\section{Fiber Sum Decomposability and Decomposition}

In this section we prove the decomposability of $X(n)$ for $2 \leq n \leq 6$ and consider their possible decompositions under the genus 2 fiber sum. 

\begin{thm}[Decomposability of $X(n)$ for $2 \leq n \leq 6$]\label{Split}

The genus 2 Lefschetz fibrations $X(n)$ for $2 \leq n \leq 6$ are all decomposable into nontrivial fiber sum of other genus 2 Lefschetz fibrations. Namely, $X(2)$ is isomorphic to an untwisted fiber sum of Matsumoto fibration on $\mathbb{S}^2 \times \mathbb{T}^2 \# 4 \CPb$ with Lefschetz fibration on $Z(0) = \CP \# 13\CPb$. Additionally, $X(3), X(4), X(5),X(6)$ are isomorphic to an untwisted fiber sum of Matsumoto fibration on $\mathbb{S}^2 \times \mathbb{T}^2 \# 4 \CPb$ with $Z(1), Z(2), Z(3),Z(4)$ respectively.
\end{thm}

\begin{proof}
As $Z(0) = \CP \# 13\CPb$ portion of the monodromy can be blown down independently (not using the conjugation $ \overset{C} {\rightarrow}$) by the above Lemma 11, it is easy to see that the untwisted fiber sum of Matsumoto's fibration on $\mathbb{S}^2 \times \mathbb{T}^2 \# 4\CPb$ having the positive relation $(\eta_1 \delta \eta_2 \eta_3)^2$ with $Z(m)$ having the positive relations of Lemma 11 for $0 \leq m \leq 4$ will give exotic copies $X(2), X(3), X(4), X(5), X(6)$ as indicated by the above monodromy factorizations of Lemma 12 which are the positive relations of $X(n)$ for $2 \leq n \leq 6$. 
\end{proof}

\begin{thm}[Unique decomposition of $X(2)$]\label{X(2)}
The genus 2 Lefschetz fibration $X(2)$ which has $n$ irreducible singular fibers and $s$ reducible singular fibers pair $(n,s) = (26,2)$ must decompose under the genus 2 fiber sum having the indecomposable summands of Matsumoto's fibration on $\mathbb{S}^2 \times \mathbb{T}^2 \# 4\CPb$ and the genus 2 Lefschetz fibration on $Z(0)=\CP \# 13\CPb$. Each summands are determined up to diffeomorphism.
\end{thm}

\begin{proof}
Let us suppose $X(2)$ decomposes into two genus 2 Lefschetz fibrations $X(2)= Y(1) \# Y(2)$ where both $Y(1), Y(2)$ are relatively minimal genus 2 Lefschetz fibrations. There are two possible cases to consider for the distribution of reducible singular fibers and hence determine the possible decompositions up to diffeomorphism.

First case is when the two reducible singular fibers distribute wholly to one of the summand (i.e.  $s=(2,0)$) where without the loss of generality, we can assume $Y(1)$ has $(n,s) = (6,2)$ and $Y(2)$ has $(n,s) = (20,0)$. Then $Y(1)$ is diffeomorphic to Lefschetz fibrations $\mathbb{S}^2 \times \mathbb{T}^2 \# 4\CPb$ by the proposition 4.1 \cite{Sato2} and $Y(2)$ is diffeomorphic to $\CP \# 13\CPb$ by above proposition on characterization of genus 2 Lefschetz fibration with 20 irreducible singular fibers. Another possibility is when $Y(1)$ has $(n,s)=(16,2)$ and $Y(2)$ has $(n,s)=(10,0)$ and we know this is impossible by the remark 5.1 of \cite{Sato2}, we know $(n,s)=(10,0)$ (the $(n,s)$ pair for $Y(2)$) cannot occur as the pair of number of singular fibers for genus 2 Lefschetz fibration. Note that these two decompositions are the only possibility for $s=(2,0)$ since  $n+2s \equiv 0 \pmod{10}$.

Second case is when $s=(1,1)$, where without the loss of generality, we can assume $Y(1)$ has $(n,s)=(8,1)$ and $Y(2)$ has $(n,s)=(18,1)$ then this is impossible by the remark 5.1 of \cite{Sato2}, as we know $(n,s)=(8,1)$ (the $(n,s)$ pair for $Y(2)$) cannot occur as the pair of number of singular fibers for genus 2 Lefschetz fibration. Note that this decomposition is the only possibility for $s=(1,1)$ since $n+2s \equiv 0 \pmod{10}$.

\end{proof}

\begin{prop}[Decompositions of $X(3)$]
The genus 2 Lefschetz fibration $X(3)$ which has $n$ irreducible singular fibers and $s$ reducible singular fibers pair $(n,s) = (24,3)$ must decompose under the genus 2 fiber sum having the summand of Matsumoto's fibration on $\mathbb{S}^2 \times \mathbb{T}^2 \# 4\CPb$ and the genus 2 Lefschetz fibration on $Z(1)=\CP \# 12\CPb$ or the genus 2 Lefschetz fibration on $\mathbb{S}^2 \times \mathbb{T}^2 \# 3\CPb$ and the genus 2 Lefschetz fibration on $Z(0)=\CP \# 13\CPb$. Each summands are determined up to diffeomorphism.
\end{prop}

\begin{proof}
Let us suppose $X(3)$ decomposes into two genus 2 Lefschetz fibrations $X(3)= Y(1) \# Y(2)$ where both $Y(1), Y(2)$ are relatively minimal genus 2 Lefschetz fibrations. There are two possible cases to consider for the distribution of reducible singular fibers and hence determine the possible decompositions up to diffeomorphism.

First case is when the three reducible singular fibers distribute wholly to one of the summand (i.e.  $s=(3,0)$) where without the loss of generality, we can assume $Y(1)$ has $(n,s)=(4,3)$ and $Y(2)$ has $(n,s)=(20,0)$. Then $Y(1)$ is diffeomorphic to Lefschetz fibrations $\mathbb{S}^2 \times \mathbb{T}^2 \# 3\CPb$ by the proposition 4.1 \cite{Sato2} and $Y(2)$ is diffeomorphic to $\CP \# 13\CPb$ by above proposition on characterization of genus 2 Lefschetz fibration with 20 irreducible singular fibers. Note that this decomposition is the only possibility for $s=(3,0)$ since $n+2s \equiv 0 \pmod{10}$.

Second case is when $s=(1,2)$, where without the loss of generality, we can assume $Y(1)$ has $(n,s)=(8,1)$ and $Y(2)$ has $(n,s)=(16,2)$ this is impossible by the remark 5.1 of \cite{Sato2}, as we know $(n,s)=(8,1)$ (the $(n,s)$ pair for $Y(1)$) cannot occur as the pair of number of singular fibers for genus 2 Lefschetz fibration.
Another possibility is when $Y(1)$ has $(n,s)=(18,1)$ then $Y(2)$ has $(n,s)=(6,2)$ we know then $Y(1)$ is diffeomorphic to $\CP \# 12\CPb$ by above proposition on characterization of genus 2 Lefschetz fibration with 18 irreducible singular fibers and 1 reducible singular fiber and $Y(2)$ is diffeomorphic to genus 2 Lefschetz fibration $\mathbb{S}^2 \times \mathbb{T}^2 \# 4\CPb$ by the proposition 4.1 \cite{Sato2}. Note that these two decompositions are the only possibility for $s=(3,1)$ since  $n+2s \equiv 0 \pmod{10}$.

\end{proof}

\rmk It is now known there exists a genus 2 Lefschetz fibration structure on $\mathbb{S}^2 \times \mathbb{T}^2 \# 3\CPb$ with seven singular fibers by the work of I. Baykur and M. Korkmaz \cite{Baykur2} (In fact, they were able to show all the possible cases of minimal genus-2 Lefschetz fibrations whose total spaces are homeomorphic to simply-connected 4-manifold with $b_2^+=3$.)

\begin{prop}[Decompositions of $X(4)$]\label{X(4)}
The genus 2 Lefschetz fibration $X(4)$ which has $n$ irreducible singular fibers and $s$ reducible singular fibers pair $(n,s) = (22,4)$ must decompose under genus 2 fiber sum having the summand of Matsumoto's fibration on $\mathbb{S}^2 \times \mathbb{T}^2 \# 4\CPb$ and the genus 2 Lefschetz fibration on $Z(2)=\CP \# 11\CPb$ or the genus 2 Lefschetz fibration on $\mathbb{S}^2 \times \mathbb{T}^2 \# 3\CPb$ and the genus 2 Lefschetz fibration on $Z(1)=\CP \# 12\CPb$. Each summands are determined up to diffeomorphism except for the $Z(2)$ which is only determined up to homeomorphism.
\end{prop}

\begin{proof}
Let us suppose $X(4)$ decomposes into two genus 2 Lefschetz fibrations $X(2)= Y(1) \# Y(2)$ where both $Y(1), Y(2)$ are relatively minimal genus 2 Lefschetz fibrations. There are three possible cases to consider for the distribution of reducible singular fibers and hence determine the possible decompositions up to homeomorphism.

First case is when the four reducible singular fibers distribute wholly to one of the summand (i.e.  $s=(4,0)$) where without the loss of generality, we can assume $Y(1)$ has $(n,s)=(2,4)$ and $Y(2)$ has $(n,s)=(20,0)$. This is impossible as $N(2,0)=\{7,8 \}$ (i.e. the minimal number of singular fibers in a genus 2 Lefschetz fibration over $\mathbb{S}^2$ is 7 or 8 \cite{Ozbagci})  whereas $Y(1)$ has 6 singular fibers. Another possibility is when $Y(1)$ has $(n,s)=(12,4)$ and $Y(2)$ has $(n,s)=(10,0)$ and we know this is also impossible by the remark 5.1 of \cite{Sato2}, as we know $(n,s)=(10,0)$ (the $(n,s)$ pair for $Y(2)$) cannot occur as the pair of number of singular fibers for genus 2 Lefschetz fibration. Note that these two decompositions are the only possibility for $s=(4,0)$ since  $n+2s \equiv 0 \pmod{10}$.

Second case is when $s=(3,1)$, where without the loss of generality, we can assume $Y(1)$ has $(n,s)=(14,3)$ and $Y(2)$ has $(n,s)=(8,1)$ this is impossible by the remark 5.1 of \cite{Sato2}, as we know $(n,s)=(8,1)$ (the $(n,s)$ pair for $Y(2)$) cannot occur as the pair of number of singular fibers for genus 2 Lefschetz fibration.
Another possibility is when $Y(1)$ has $(n,s)=(4,3)$ and $Y(2)$ has $(n,s)=(18,1)$ then we know $Y(1)$ is diffeomorphic to genus 2 Lefschetz fibration $\mathbb{S}^2 \times \mathbb{T}^2 \# 3\CPb$ by the proposition 4.1 \cite{Sato2} and $Y(2)$ is diffeomorphic to $\CP \# 12\CPb$ by the above proposition on characterization of genus 2 Lefschetz fibration with 18 irreducible singular fibers and 1 reducible singular fiber. Note that these two decompositions are the only possibility for $s=(3,1)$ since  $n+2s \equiv 0 \pmod{10}$.

Third case is when $s=(2,2)$, where without the loss of generality, we can assume $Y(1)$ has $(n,s)=(6,2)$ and $Y(2)$ has $(n,s)=(16,2)$ we know then $Y(1)$ is diffeomorphic to genus 2 Lefschetz fibration $\mathbb{S}^2 \times \mathbb{T}^2 \# 4\CPb$ by the proposition 4.1 \cite{Sato2} and $Y(2)$ is homeomorphic to $\CP \# 11\CPb$. Note that this decomposition is the only possibility for $s=(2,2)$ since  $n+2s \equiv 0 \pmod{10}$.

\end{proof}

\begin{prop}[Decompositions of $X(5)$]\label{X(5)}
The genus 2 Lefschetz fibration $X(5)$ which has $n$ irreducible singular fibers and $s$ reducible singular fibers pair $(n,s) = (20,5)$ must decompose under genus 2 fiber sum having the summands of Matsumoto's fibration on $\mathbb{S}^2 \times \mathbb{T}^2 \# 4\CPb$ and the genus 2 Lefschetz fibration on $Z(3)=\CP \# 10\CPb$ or the genus 2 Lefschetz fibration on $\mathbb{S}^2 \times \mathbb{T}^2 \# 3\CPb$ and the genus 2 Lefschetz fibration on $Z(2)=\CP \# 11\CPb$. The $Z(3)$ and $Z(2)$ genus 2 Lefschetz fibration summands are determined up to homeomorphism. The $\mathbb{S}^2 \times \mathbb{T}^2 \# 3\CPb$ and $\mathbb{S}^2 \times \mathbb{T}^2 \# 4\CPb$ genus 2 Lefschetz fibration summands are determined up to diffeomorphism.
\end{prop}

\begin{proof}
Let us suppose $X(5)$ decomposes into two genus 2 Lefschetz fibrations $X(5)= Y(1) \# Y(2)$ where both $Y(1), Y(2)$ are relatively minimal genus 2 Lefschetz fibrations. There are three possible cases to consider for the distribution of reducible singular fibers and hence determine the possible decompositions up to homeomorphism.

First case is when the five reducible singular fibers distribute wholly to one of the summand (i.e.  $s=(5,0)$) where without the loss of generality, we can assume $Y(1)$ has $(n,s)=(0,5)$ and $Y(2)$ has $(n,s)=(20,0)$. This is impossible as there is no hyperelliptic Lefschetz fibration over $\mathbb{S}^2$ with only reducible singular fibers (cf. \cite{Ozbagci}) whereas $Y(1)$ has 5 reducible singular fibers only. Another possibility is when $Y(1)$ has $(n,s)=(10,5)$ and $Y(2)$ has $(n,s)=(10,0)$. This is impossible by the remark 5.1 of \cite{Sato2}, as we know $(n,s)=(10,0)$ (the $(n,s)$ pair for $Y(2)$) cannot occur as the pair of number of singular fibers for genus 2 Lefschetz fibration. Note that these two decompositions are the only possibility for $s=(5,0)$ since  $n+2s \equiv 0 \pmod{10}$.

Second case is when $s=(4,1)$, where without the loss of generality, we can assume $Y(1)$ has $(n,s)=(12,4)$ and $Y(2)$ has $(n,s)=(8,1)$ this is impossible by the remark 5.1 of \cite{Sato2}, as we know $(n,s)=(8,1)$ (the $(n,s)$ pair for $Y(2)$) cannot occur as the pair of number of singular fibers for genus 2 Lefschetz fibration. 
Another possibility is when $Y(1)$ is has $(n,s)=(2,4)$ and $Y(2)$ has $(n,s)=(18,1)$ This is impossible as $N(2,0)=\{7,8 \}$ (i.e. the minimal number of singular fibers in a genus 2 Lefschetz fibration over $\mathbb{S}^2$ is 7 or 8) \cite{Ozbagci} whereas $Y(1)$ has 6 singular fibers. Note that these two decompositions are the only possibility for $s=(4,1)$ since  $n+2s \equiv 0 \pmod{10}$.

Third case is when $s=(2,3)$, where without the loss of generality, we can assume $Y(1)$ has $(n,s)=(6,2)$ and $Y(2)$ has $(n,s)=(14,3)$ we know then $Y(1)$ is diffeomorphic to genus 2 Lefschetz fibration $\mathbb{S}^2 \times \mathbb{T}^2 \# 4\CPb$ by the proposition 4.1 \cite{Sato2} and $Y(2)$ is homeomorphic to $\CP \# 10\CPb$. Another possibility is when $Y(1)$ has $(n,s)=(4,3)$ and $Y(2)$ has $(n,s)=(16,2)$ then we know $Y(1)$ is diffeomorphic to genus 2 Lefschetz fibration $\mathbb{S}^2 \times \mathbb{T}^2 \# 3\CPb$ by the proposition 4.1 \cite{Sato2} and $Y(2)$ is homeomorphic to $\CP \# 11\CPb$. Note that these two decompositions are the only possibility for $s=(2,3)$ since  $n+2s \equiv 0 \pmod{10}$.

\end{proof}

\begin{prop}[Decompositions of $X(6)$]\label{X(6)}
The genus 2 Lefschetz fibration $X(6)$ which has $n$ irreducible singular fibers and $s$ reducible singular fibers pair $(n,s) = (18,6)$ must decompose under genus 2 fiber sum having the summand of Matsumoto's fibration on $\mathbb{S}^2 \times \mathbb{T}^2 \# 4\CPb$ and the genus 2 Lefschetz fibration on $Z(4)=\CP \# 9\CPb$ or the genus 2 Lefschetz fibration on $\mathbb{S}^2 \times \mathbb{T}^2 \# 3\CPb$ and the genus 2 Lefschetz fibration on $Z(3)=\CP \# 10\CPb$. The $Z(4)$ and $Z(3)$ genus 2 Lefschetz fibration summands are determined up to homeomorphism. The $\mathbb{S}^2 \times \mathbb{T}^2 \# 3\CPb$ and $\mathbb{S}^2 \times \mathbb{T}^2 \# 4\CPb$ genus 2 Lefschetz fibration summands are determined up to diffeomorphism.
\end{prop}

\begin{proof}
Let us suppose $X(6)$ decomposes into two genus 2 Lefschetz fibrations $X(2)= Y(1) \# Y(2)$ where both $Y(1), Y(2)$ are relatively minimal genus 2 Lefschetz fibrations. There are four possible cases to consider for the distribution of reducible singular fibers and hence determine the possible decompositions up to homeomorphism.

First case is when the six reducible singular fibers distribute wholly to one of the summand (i.e.  $s=(6,0)$) where without the loss of generality, we can assume $Y(1)$ has $(n,s)=(8,6)$ and $Y(2)$ has $(n,s)=(10,0)$. This is impossible by the remark 5.1 of \cite{Sato2}, as we know $(n,s)=(10,0)$ cannot occur as the pair of number of singular fibers for Lefschetz fibration. Note that this decomposition is the only possibility for $s=(6,0)$ since  $n+2s \equiv 0 \pmod{10}$.

Second case is when $s=(5,1)$, where without the loss of generality, we can assume $Y(1)$ has $(n,s)=(0,5)$ and $Y(2)$ has $(n,s)=(18,1)$ this is impossible by the remark 5.1 of \cite{Sato2}, as we know there is no hyperelliptic Lefschetz fibration over $\mathbb{S}^2$ with only reducible singular fibers \cite{Ozbagci} whereas $Y(1)$ has 5 reducible singular fibers only. Another possibility is when $Y(1)$ has $(n,s)=(10,5)$ and $Y(2)$ has $(n,s)=(8,1)$ this is impossible by the remark 5.1 of \cite{Sato2}, as we know $(n,s)=(8,1)$ (the $(n,s)$ pair for $Y(2)$) cannot occur as the pair of number of singular fibers for genus 2 Lefschetz fibration. Note that these two decompositions are the only possibility for $s=(5,1)$ since  $n+2s \equiv 0 \pmod{10}$.

Third case is when $s=(4,2)$, where without the loss of generality, we can assume $Y(1)$ has $(n,s)=(2,4)$ then $Y(2)$ has $(n,s)=(16,2)$ this is impossible as $N(2,0)=\{7,8 \}$ (i.e. the minimal number of singular fibers in a genus 2 Lefschetz fibration over $\mathbb{S}^2$ is 7 or 8) \cite{Ozbagci} whereas $Y(1)$ has 6 singular fibers. Another possibility is when $Y(1)$ has $(n,s)=(12,4)$ and $Y(2)$ has $(n,s)=(6,2)$ then we know $Y(1)$ is homeomorphic to $\CP \# 9\CPb$ and $Y(2)$ is diffeomorphic to genus 2 Lefschetz fibration $\mathbb{S}^2 \times \mathbb{T}^2 \# 4\CPb$ by the proposition 4.1 \cite{Sato2}. Note that these two decompositions are the only possibility for $s=(4,2)$ since  $n+2s \equiv 0 \pmod{10}$.

Fourth case is when $s=(3,3)$, where without the loss of generality, we can assume $Y(1)$ has $(n,s)=(4,3)$ and $Y(2)$ has $(n,s)=(14,3)$ then we know $Y(1)$ is diffeomorphic to genus 2 Lefschetz fibration $\mathbb{S}^2 \times \mathbb{T}^2 \# 3\CPb$ by the proposition 4.1 \cite{Sato2} and $Y(2)$ is homeomorphic to $\CP \# 10\CPb$. Note that this decomposition is the only possibility for $s=(3,3)$ since  $n+2s \equiv 0 \pmod{10}$.
\end{proof}

\rmk Even though one can easily see indecomposability of $X(0)$ and $X(1)$ from non-minimality (cf. \cite{Stipsicz3}) one can also prove $X(0)$ and $X(1)$ are indecomposable under the genus 2 fiber sum by the similar reasoning on the possible pairs of $(n,s)$ for the summands. 

\par As $X(0)$ has 30 irreducible singular fibers $(n,s) = (30,0)$ if it were to decompose into two genus 2 Lefschetz fibrations $X(2)= Y(1) \# Y(2)$ where both $Y(1), Y(2)$ are relatively minimal genus 2 Lefschetz fibrations there is only one possible case of decomposition. Since $n+2s \equiv 0 \pmod{10}$, without the loss of the generality $Y(1)$ has $(n,s) = (10,0)$ and $Y(2)$ has $(n,s) = (20,0)$ this is impossible by the remark 5.1 of \cite{Sato2}, we know $(n,s)=(10,0)$ cannot occur as the pair of number of singular fibers for Lefschetz fibration.

\par Similarly for $X(1)$ which has $(n,s)=(28,1)$ we can consider possible pairs of $(n,s)$ for both $Y(1), Y(2)$. There are only two possible cases to consider namely when $Y(1)$ has $(n,s)=(8,1)$ while $Y(2)$ has $(n,s)=(20,0)$ and another possible case when $Y(1)$ has $(n,s)=(18,1)$ while $Y(2)$ has $(n,s)=(10,0)$. Both cases are impossible by the remark 5.1 of \cite{Sato2}, we know $(n,s)=(10,0)$ and $(n,s)=(8,1)$ cannot occur as the pair of number of singular fibers for Lefschetz fibration and thus such decomposition is impossible.\\

\rmk Similar reasoning on the possible pairs of $(n,s)$ for the summands applies also to the Endo-Gurtas examples such as $Z(m)$ for $0 \leq m \leq 3$ to show indecomposability.  

\par As $Z(0)$ has 20 irreducible singular fibers $(n,s) = (20,0)$ if it were to decompose into two genus 2 Lefschetz fibrations $X(2)= Y(1) \# Y(2)$ where both $Y(1), Y(2)$ are relatively minimal genus 2 Lefschetz fibrations there is only one possible case of decomposition. Since $n+2s \equiv 0 \pmod{10}$, without the loss of the generality $Y(1)$ has $(n,s) = (10,0)$ and $Y(2)$ has $(n,s) = (10,0)$ this is impossible by the remark 5.1 of \cite{Sato2}, we know $(n,s)=(10,0)$ cannot occur as the pair of number of singular fibers for Lefschetz fibration.

\par Similarly for $Z(1)$ which has $(n,s)=(18,1)$ we can consider possible pairs of $(n,s)$ for both $Y(1), Y(2)$. There is only one possible case to consider namely when $Y(1)$ has $(n,s)=(8,1)$ while $Y(2)$ has $(n,s)=(10,0)$ whereas we know both are ruled out of existence by the remark 5.1 of \cite{Sato2}. 

\par As for $Z(2)$ there are only two possible cases to consider namely when $Y(1)$ has $(n,s)=(6,2)$ while $Y(2)$ has $(n,s)=(10,0)$ and another possible case when $Y(1)$ has $(n,s)=(8,1)$ while $Y(2)$ has $(n,s)=(8,1)$. Both cases are impossible as  the remark 5.1 of \cite{Sato2}, we know $(n,s)=(10,0)$ and $(n,s)=(8,1)$ cannot occur as the pair of number of singular fibers for Lefschetz fibration and thus such decomposition is impossible. 

\par Finally for the $Z(3)$, there are again only two possible cases to consider namely when $Y(1)$ has $(n,s)=(4,3)$ while $Y(2)$ has $(n,s)=(10,0)$ and another possible case when $Y(1)$ has $(n,s)=(16,2)$ while $Y(2)$ has $(n,s)=(8,1)$. Both cases are again impossible by the remark 5.1 of \cite{Sato2}, we know $(n,s)=(10,0)$ and $(n,s)=(8,1)$ cannot occur as the pair of number of singular fibers for Lefschetz fibration and thus such decomposition is impossible.\\

\par Interestingly, it is impossible to rule out the decomposability of $Z(4)$ as suggested by Endo-Gurtas,

\begin{prop}[Decompositions of $Z(4)$]\label{Z(4)}
$Z(4)$ which has $n$ irreducible singular fibers and $s$ reducible singular fibers pair $(n,s) = (12,4)$ if it were to decompose it must decompose under genus 2 fiber sum having the indecomposable summands of Matsumoto's fibration on $\mathbb{S}^2 \times \mathbb{T}^2 \# 4\CPb$. The summands are determined up to diffeomorphism.
\end{prop}

\begin{proof}
Let us suppose $Z(4)$ decomposes into two genus 2 Lefschetz fibrations $X(2)= Y(1) \# Y(2)$ where both $Y(1), Y(2)$ are relatively minimal genus 2 Lefschetz fibrations. There are three possible cases to consider for the distribution of reducible singular fibers and hence determine the possible decompositions up to diffeomorphism.

First case is when the four reducible singular fibers distribute wholly to one of the summand (i.e.  $s=(4,0)$) where without the loss of generality, we can assume $Y(1)$ has $(n,s)=(2,4)$ and $Y(2)$ has $(n,s)=(10,0)$. This is impossible as $N(2,0)=\{7,8 \}$ (i.e. the minimal number of singular fibers in a genus 2 Lefschetz fibration over $\mathbb{S}^2$ is 7 or 8) \cite{Ozbagci} whereas $Y(1)$ has 6 singular fibers. It is also impossible by the remark 5.1 of \cite{Sato2}, as we know $(n,s)=(10,0)$ (the $(n,s)$ pair for $Y(2)$) cannot occur as the pair of number of singular fibers for genus 2 Lefschetz fibration. Note that this is the only possible decomposition case to consider for $s=(4,0)$ since  $n+2s \equiv 0 \pmod{10}$.

Second case is when $s=(3,1)$, where without the loss of generality, we can assume $Y(1)$ has $(n,s)=(4,3)$ and $Y(2)$ has $(n,s)=(8,1)$ this is impossible by the remark 5.1 of \cite{Sato2}, as we know $(n,s)=(8,1)$ (the $(n,s)$ pair for $Y(2)$) cannot occur as the pair of number of singular fibers for genus 2 Lefschetz fibration. Note that this is the only possible decomposition case to consider for $s=(3,1)$ since  $n+2s \equiv 0 \pmod{10}$.

Third case is when $s=(2,2)$, where without the loss of generality, we can assume $Y(1)$ has $(n,s)=(6,2)$ and $Y(2)$ has $(n,s)=(6,2)$ we know then $Y(1)$ and $Y(2)$ must be diffeomorphic to genus 2 Lefschetz fibration $\mathbb{S}^2 \times \mathbb{T}^2 \# 4\CPb$ by the proposition 4.1 \cite{Sato2}.
\end{proof}

\par As it is still not known whether or not $Z(4)$ in our article or $E$ in Endo-Gurtas are actually decomposable into the two genus 2 Lefschetz fibrations to begin with this decomposition result alone does not fully answer the question asked by Endo-Gurtas \cite{EG}.

\section*{Acknowledgments} I am grateful to Anar Akhmedov and Refik \.{I}nan\c{c} Baykur for suggesting this problem and for many useful discussions and ideas. I am also grateful to Tian-Jun Li, Andr\'{a}s I. Stipsicz and Chuen-Ming Michael Wong for helpful conversations.

\bibliographystyle{mrl}

\end{document}